\documentclass[a4paper]{amsart}

\usepackage[a4paper,width={16cm},left=2.5cm,bottom=3cm, top=3cm]{geometry}
\usepackage{enumerate}

\usepackage[utf8]{inputenc}
\usepackage[english]{babel}

\usepackage{a4wide}
\usepackage{xcolor}
\usepackage{amsmath}
\usepackage{amssymb}
\usepackage{amsfonts}
\usepackage{amsthm,graphicx}

\usepackage{hyperref}

\newtheorem{theorem}{Theorem}[section]
\newtheorem{lemma}[theorem]{Lemma}
\newtheorem{corollary}[theorem]{Corollary}

\newtheorem{proposition}[theorem]{Proposition}
\theoremstyle{definition}

\newtheorem{remark} [theorem] {Remark}

\newcommand{\la}{\lambda}
\newcommand{\norm}[1]{\left\lVert#1\right\rVert}
\newcommand{\pd}[2]{\frac{\partial#1}{\partial#2}}

\newcommand{\R}{\mathbb{R}}

\newcommand{\e}{\epsilon}
\newcommand{\co}{{\rm{const}}}

\begin{document}

\title[Unique continuation from a crack's tip under Neumann boundary conditions]
{Unique continuation from a crack's tip \\ under Neumann boundary conditions}
\date{\today}

\author{Veronica Felli and Giovanni Siclari}

\address[V. Felli and G. Siclari]{Dipartimento di Matematica
  e Applicazioni
\newline\indent
Universit\`a degli Studi di Milano - Bicocca
\newline\indent
Via Cozzi 55, 20125, Milano, Italy}
\email{veronica.felli@unimib.it, g.siclari2@campus.unimib.it}

\begin{abstract}
  We derive local asymptotics of solutions to second order elliptic
  equations at the edge of a $(N-1)$-dimensional crack, with homogeneous Neumann
  boundary conditions prescribed on both sides of the crack.
  A combination of  blow-up analysis and monotonicity arguments
  provides a classification of all possible asymptotic homogeneities
  of solutions at the crack's tip, together with a a strong
  unique continuation principle.
\end{abstract}

\maketitle



\section{Introduction}
In this paper we establish a strong unique continuation principle and
analyse the asymptotic behaviour of solutions,
from the edge of a flat
crack $\Gamma$, for the
following elliptic problem with homogeneous Neumann boundary
conditions on both sides of the crack
\begin{equation}\label{op}
  \begin{cases}
    -\Delta u=f u, &\text{in } B_R \setminus \Gamma,  \\[5pt]
    \dfrac{\partial^+u}{\partial\nu^+}=\dfrac{\partial^-u}{\partial\nu^-}=0,
    &\text{on } \Gamma,
  \end{cases}
\end{equation}
where 
\[
  B_R=\{x \in \R^N: |x| <R\}\subset \R^N,  \quad  N \ge 2,
\] 
$\Gamma$ is a closed subset of $\R^{N-1}\times\{0\}$ with
  $C^{1,1}$-boundary, and
  the potential $f$  satisfies either  assumption \eqref{h1} or  assumption  \eqref{h2} below.
The boundary operators
  $\frac{\partial^+}{\partial\nu^+}$ and 
  $\frac{\partial^-}{\partial\nu^-}$ in  \eqref{op} are defined as
\begin{equation*}
  \dfrac{\partial^+u}{\partial\nu^+}:=-\frac{\partial}{\partial
    x_N}\left(u\big|_{B_R^+}\right)\quad\text{and}\quad
  \dfrac{\partial^-u}{\partial\nu^-}:=\frac{\partial}{\partial
    x_N}\left(u\big|_{B_R^-}\right),
\end{equation*}
where we are denoting, for all $r>0$,
\begin{equation}\label{Br}
  B^+_r:=\{(x',x_{N-1},x_{N})\in B_r: x_{N}>0\}, \quad B^-_r:=\{(x',x_{N-1},x_N)\in B_r: x_N<0\},	
\end{equation}
being 
the total variable $x\in\R^N$ written as
$x=(x',x_{N-1},x_{N})\in\R^{N-2}\times\R\times\R$.

The interest in elliptic problems in domains with cracks is motivated
by elasticity theory, see e.g. \cite{KLH,DMOT}.  In particular, in crack problems, the coefficients of the asymptotic expansion of
solutions near the crack’s tip are related to the so called
\emph{stress intensity factor}, see \cite{DMOT}. We refer to
\cite{CD,costabel2003asymptotics,DW}  and references therein for the
study of the 
behaviour of solutions at the edge of a cut.

We recall that a family of functions $\mathcal{F}=\{f_i\}_{i \in I}$,
with $f_i:A \to \R$, $A \subseteq \R^N$, satisfies the \emph{strong unique
continuation} property if 
no function in $\mathcal F$, besides possibly the trivial null
function, has a zero of infinite order at any point $x_0 \in A$.
The first significant contribution to the study of strong
  unique continuation for second order elliptic equations was given by
  Carleman in \cite{carleman1939probleme} for bounded potentials  in
  dimension $2$, by means of weighted a priori inequalities. The
  so-called \emph{Carleman estimates} are still today one of
  the main techniques 
  used in this research field and have been adapted by many
  authors to generalize Carleman's results and prove unique
  continuation for more general classes of elliptic 
  equations; among the numerous contributions in this area we mention
  \cite{Aronszajn,jerisonkenig1985,sogge,wolff}  and in particular
  \cite{koch2001carleman}, where  strong
  unique continuation is established under sharp scale invariant
  assumptions on the
  potentials. Garofalo and Lin  developed  in
  \cite{garofalo1986monotonicity} an alternative approach to the study of unique
  continuation, based on local doubling
  inequalities,  which are in turn deduced by the monotonicity of an Almgren
  type frequency function, see \cite{almgren1983q}. In the present
  paper we follow this latter approach and study the Almgren frequency
  function $\mathcal N$
  around the point $0$ lying on the edge of the crack, defined as the ratio between
the local energy function
\[
  E(r):=\frac{1}{r^N}\int_{B_r\setminus\Gamma}(|\nabla
  u|^2-fu^2) \, dx
\]
and the local mass or height
\[
  H(r):=\frac{1}{r^{N-1}}\int_{\partial B_r}u^2 \, d\sigma,
\]
i.e.
\[
  \mathcal{N}(r):=\frac{E(r)}{H(r)}.
\]
 
The boundedness of the frequency function $\mathcal N$ will imply 
a strong unique continuation principle from the edge
of $\Gamma$. Furthermore, the
  monotonicity properties of the quotient $\mathcal N$ will allow us
  to obtain energy estimates, which will be combined with a blow-up
  analysis for scaled solutions to prove that any $u\in
  H^1(B_R\setminus\Gamma)$ weakly solving 
  \eqref{op} behaves, asymptotically at the edge of the crack $\Gamma$, as a
  homogeneous function with half-integer degree of homogeneity.
We mention that an analogous procedure for classifying
all possible asymptotic homogeneity degrees of solutions by monotonicity formula
and blow-up analysis was introduced in 
\cite{felli2010asymptotic,felli2012note,felli2010behavior} for equations with singular
potentials  and adapted
to domains with corners in \cite{felli2013almgren}.

The derivation of a monotonicity formula around a
boundary point presents some additional difficulties with respect to
the interior case,  due to the 
role that the regularity and the geometry of the domain
may play.

   Among papers dealing with unique continuation from the boundary
   under homogeneous Dirichlet conditions  we
cite \cite{adolfsson1997c1,Adolf-Kenig,felli2013almgren,
  Kukavica-Nystrom}. Instead, for Neumann problems, we refer to 
\cite{adolfsson1997c1} and 
\cite{tao2005boundary} for the  
homogeneous case and to \cite{dipierro2020unique} for unique
continuation from the vertex of a cone under  non-homogeneous
Neumann  
conditions. We also mention that unique continuation from Dirichlet-Neumann junctions
for planar mixed boundary value problems was established in \cite{FFFN}.

The  high non-smoothness of the domain $B_R \setminus
  \Gamma$ at points on the edge of the crack causes two kinds of
  difficulties in the proof of the Pohozaev type identity which is needed to estimate the 
derivative of the Almgren frequency function, see Proposition
\ref{N'}.  A first difficulty is a lack of regularity that can
  prevent us from integrating  Rellich-Ne\u cas identities of type \eqref{Rellich-Necas}. A second
  issue is related to
 the interference with
  the geometry of the crack, which manifests in the form of extra terms,
  produced by  integration by parts, which could be problematic to 
  estimate.

  In \cite{de2021unique}, where homogeneous Dirichlet
    conditions on the crack are considered, this latter difficulty is
    overcome by assuming a local star-shapedness condition for the
    cracked domain, which forces the  extra terms 
  produced by  integration by parts to have a sign favourable to 
  the desired estimates. The problem produced by lack of regularity
  is instead solved in \cite{de2021unique}  by approximating  $B_r
\setminus \Gamma$ with a sequence  of smooth domains 
$\Omega_{n,r} \subset B_r$ and constructing approximating
  problems in $\Omega_{n,r}$, whose solutions $u_n$ converge in $H^1(B_r)$
  to the solution of the original cracked  problem for $r
  \in (0,R)$ small enough. Each 
function $u_n$ is sufficiently regular to satisfy a Pohozaev type
identity, in which it is possible to pass to the limit as $n \to 
\infty$, obtaining  the inequality needed to estimate the
  derivative of the Almgren frequency function.

  In the present paper
we use a similar approximation technique, which however entails
additional difficulties and requires substantial modifications due to
the Neumann boundary conditions.
In particular, the existence of an extension operator  for
Sobolev functions on  $\Omega_n$, uniform with respect to $n$,
which is obvious under  Dirichlet boundary conditions,  turns out to be
more delicate in the Neumann case, see Proposition
\ref{pon ex}. Furthermore the different boundary conditions produce
remainder terms with different signs, requiring a
modified profile for the approximating domains, see Section
\ref{sec:appr-probl}.

Unlike \cite{de2021unique}, we do not require any geometric
star-shapedness condition on
 the crack  $\Gamma$, limiting ourselves to a 
 $C^{1,1}$-regularity assumption,  see \eqref{condition_g1} below.
The removal of the star-shapedness condition assumed in
\cite{de2021unique} requires a more sophisticated
monotonicity formula, developed for the auxiliary problem
\eqref{straightenedproblem}, obtained after   straightening the  crack
$\Gamma$ with a diffeomorphism introduced in \cite{adolfsson1997c1}
and used, with a similar purpose,  in \cite{de2021strong} for fractional elliptic equations, see Section \ref{sec:boundary}.
 The effect of this transformation straightening the crack
  is the appearance of a variable coefficient matrix in the
  divergence-form elliptic operator, with a consequent
  adaption of the definition of the energy $E$ and the height $H$ 
in \eqref{H} and \eqref{E}.

To state the main results of this paper, we introduce now our
assumptions on the crack $\Gamma$ and the potential $f$.
We suppose that $\Gamma$ is a closed set   of the form 
\begin{equation}\label{Gamma 2}
\Gamma:=\{(x_1,0):x_1 \in [0,+\infty)\}\quad \text{ if } N=2
\end{equation}
and
\begin{equation} \label{Gamma}
\Gamma:=\{(x',x_{N-1},0)\in \R^N: g(x')\le x_{N-1}\}\quad \text{ if } N \ge 3,
\end{equation}
where
 \begin{equation} \label{condition_g1}
   g:\R^{N-2} \rightarrow \R, \quad g \in C^{1,1}(\R^{N-2}),
 \end{equation}
and 
\begin{equation} \label{condition_g2}
  	g(0)=0,\quad \nabla g(0)=0.
\end{equation}
Assumption \eqref{condition_g2} is not
  restrictive, being  a free consequence of an appropriate choice of the Cartesian
coordinate system. We are going to study the behaviour of solutions to
\eqref{op} near $0$, which belongs to the edge of the crack $\Gamma$
defined in \eqref{Gamma 2}--\eqref{Gamma}.

Furthermore we  assume that
$f:B_R\to\R$ is a measurable function for which
there exists
$\epsilon\in (0,1)$ such that either
		\begin{equation}
			\tag{H1} \label{h1}
			f \in W^{1,\frac{N}{2}+\epsilon}(B_R\setminus \Gamma),
		\end{equation}
		or
		\begin{equation}
			\tag{H2}\label{h2}
			N\ge 3\quad  \text{and}\quad  |f(x)| \le c |x|^{-2+2\epsilon}
			\quad \text{for some }  c > 0\text{ and for all }x \in B_R.
		\end{equation}
For every closed set $K\subseteq \R^N\times \{0\}$ 
and $r >0$, we define the functional space
$H^1_{0,\partial B_r}(B_r\setminus K)$ as the closure in $H^1(B_r\setminus K)$ of the set
		\begin{equation*}
			\{v\in H^1(B_r\setminus K):
			v=0 \text{ in a neighbourhood of }\partial B_r\}.
		\end{equation*}
 A weak solution to \eqref{op} is a function $u \in H^1(B_R \setminus \Gamma)$ such that
		\begin{equation*}
			\int_{B_R\setminus\Gamma}(\nabla u\cdot \nabla \phi -fu \phi) \, dy=0,	
		\end{equation*}
for all $\phi \in H_{0,\partial B_R}^1(B_R \setminus \Gamma)$.

The following 
unique continuation principle for solutions to \eqref{op} is our main result.
\begin{theorem}\label{t:ucp}
	Let $u$ be a weak solution to \eqref{op}  with
          $\Gamma$ as in   \eqref{Gamma
  2}--\eqref{Gamma} and $f$ satisfying either \eqref{h1} or
 \eqref{h2}. If $u(x)=O(|x|^k)$ as $|x| \to  0^+$ for all $k \in
\mathbb{N}$,
then 	$u\equiv0$ in $B_R$.
\end{theorem}

In Theorem \ref{u-asymptotic} we provide a classification of blow-up
limits  in terms of the eigenvalues of 
the following
problem
\begin{equation}\label{sphereprob}
  \begin{cases}
    -\Delta_{\mathbb{S}^{N-1}} \psi =\mu \psi, &\text{ on } \mathbb{S}^{N-1} \setminus \Sigma,\\
    \dfrac{\partial^+\psi}{\partial\nu^+}=\dfrac{\partial^-\psi}{\partial\nu^-}=0,
    &\text{ on } \Sigma,
  \end{cases}
\end{equation}
on the unit $(N-1)$-dimensional sphere $\mathbb{S}^{N-1}:=\{x \in
\R^N: |x|=1\}$ 
with a cut on the
half-equator
\begin{equation*}
  \Sigma:=\{(x',x_{N-1},0)\in \mathbb{S}^{N-1}: x_{N-1} \ge 0\},
\end{equation*}
where, letting  $e_N:=(0,\dots, 1)$, 
\begin{align*}
	&\mathbb{S}^{N-1}_+:=\left\{(x',x_{N-1},x_N)\in
   \mathbb{S}^{N-1}: x_N >0 \right\},\quad 
	\mathbb{S}^{N-1}_-:=\left\{(x',x_{N-1},x_N)\in \mathbb{S}^{N-1}: x_N <0 \right\},
\end{align*} 
the boundary operators $\frac{\partial^\pm}{\partial\nu^\pm}$ are
defined as
\begin{equation*}
  \dfrac{\partial^+\psi}{\partial\nu^+}:=
  -\nabla_{\mathbb{S}_+^{N-1}}\left(\psi\big|_{\mathbb{S}^{N-1}_+}\right) \cdot e_N \quad\text{and}\quad
  \dfrac{\partial^-\psi}{\partial\nu^-}:=
  \nabla_{\mathbb{S}_-^{N-1}}\left(\psi\big|_{\mathbb{S}^{N-1}_-}\right)\cdot e_N,
\end{equation*}
see Section \ref{sec:neum-eigenv-mathbbsn} for the
  weak formulation of \eqref{sphereprob}.
In Section \ref{sec:neum-eigenv-mathbbsn} we prove that the set of the
eigenvalues of \eqref{sphereprob} is $\{\mu_k: k \in \mathbb{N}\}$
where 
\[
  \mu_k=\frac{k(k+2N-4)}{4}, \quad k \in \mathbb{N}.
\]
As a consequence of the classification of blow-up limits, we obtain
the following
unique continuation result from the edge with respect to crack points.
\begin{theorem}\label{t:ucp-crack}
  Let $u$ be a weak solution to \eqref{op} with $\Gamma$ as in
  \eqref{Gamma 2}--\eqref{Gamma} and $f$ satisfying either \eqref{h1}
  or \eqref{h2}.  Let us also assume that $u$ vanishes at $0$ at any
  order with respect to crack points, namely that either
  $\mathop{\rm Tr}^+_\Gamma u(z)=O(|z|^k)$ as $|z| \to 0^+$,
  $z\in\Gamma$, for all $k \in \mathbb{N}$ or
  $\mathop{\rm Tr}^-_\Gamma u(z)=O(|z|^k)$ as $|z| \to 0^+$,
  $z\in\Gamma$, for all $k \in \mathbb{N}$, where
  $\mathop{\rm Tr}^+_\Gamma u$, respectively
  $\mathop{\rm Tr}^-_\Gamma u$, denotes the trace of $u\big|_{B_R^+}$,
  respectively $u\big|_{B_R^-}$, on $\Gamma$.  Then $u\equiv0$ in
  $B_R$.
\end{theorem}

If $N\geq3$, a combination of the blow-up analysis with an expansion in Fourier series with respect to a
orthonormal basis made of eigenfunctions of \eqref{sphereprob}, allows us to classify the possible asymptotic homogeneity degrees
  of solutions at $0$.

\begin{theorem} \label{main-theorem} Let $N \ge3$ and let
  $u \in H^1(B_R\setminus \Gamma)$, $u\not\equiv0$,  be a non-trivial
  weak solution to
  \eqref{op}, with $\Gamma$ defined in  \eqref{Gamma
  2}--\eqref{Gamma} and $f$ satisfying either assumption \eqref{h1} or
  assumption \eqref{h2}. Then there exist $k_0\in \mathbb{N}$ and an
  eigenfunction $Y$ of problem \eqref{sphereprob}, associated to
  the eigenvalue $\mu_{k_0}$, such that, 
 letting
\[
\Phi(x):=|x|^{\frac{k_0}{2}}  Y\left(\frac{x}{|x|}\right),
\]
we have that
\begin{align*}
	&\lambda^{-\frac{k_0}{2}}u(\lambda \cdot) \to
          \Phi\quad\text{and}\quad \lambda^{1-\frac{k_0}{2}}
\left(\nabla_{B_R\setminus\Gamma}u\right)(\lambda
     \cdot)
     \to \nabla_{\R^N\setminus\tilde\Gamma}\Phi \quad
     \text{in } L^2(B_1)
\end{align*}
as $\la \to 0^+$, 
where 
\begin{equation}\label{Gammatilde}
	\tilde{\Gamma}:=\left\{ x=(x',x_{N-1},0) \in \R^N:x_{N-1} \ge 0\right\}
\end{equation}
and 
$\nabla_{B_R\setminus\Gamma}$ and $\nabla_{\R^N\setminus\tilde\Gamma}$
denote the distributional gradients in $B_R\setminus\Gamma$ and
$\R^N\setminus\tilde\Gamma$ respectively.

\end{theorem}
A more precise version of Theorem \ref{main-theorem}, 
  relating $k_0$ to the limit of a frequency function and
  characterizing the eigenfunction $Y$, will be proved in Section
  \ref{sec:colorr-height-funct}, see Theorem \ref{main-theorem-precise}.

The paper is organized as follows. In Section \ref{sec:boundary} an
equivalent problem in a domain with a straightened crack is
constructed. Sections \ref{sec:trac-embedd-space} contains  some trace and embedding
inequalities for the space  $H^1(B_{r}\setminus
\tilde\Gamma)$. Section \ref{sec:appr-probl} is devoted to the
construction of the approximating problems. In
Section \ref{sec:almgr-type-freq} we develop the
monotonicity argument, which is first used to prove Theorem \ref{t:ucp}
and later, in Section
\ref{sec:blow-up-analysis}, to perform a blow-up analysis and prove
Theorem \ref{t:ucp-crack}, taking into
account the structure of the spherical eigenvalue problem
\eqref{sphereprob} studied in Section
\ref{sec:neum-eigenv-mathbbsn}. Finally Theorem \ref{main-theorem} is
proved in Section \ref{sec:colorr-height-funct}.

\section{An equivalent problem with straightened crack}
\label{sec:boundary}
In this section we straighten the boundary of the crack in
  a neighbourhood of $0$.  
If $N \ge 3$ we use the local
diffeomorphism $F$ defined in \cite[Section 2]{de2021strong}, see also
\cite{adolfsson1997c1};
for the sake of clarity and completeness we summarize
its properties in Propositions \ref{diffeomorphism}
  and \ref{beta-properties} below, referring to  \cite[Section
  2]{de2021strong} for their proofs. If $N=2$, the crack is a segment and we simply take
$F=\mathop{\rm{Id}}$, where $\mathop{\rm{Id}}$ is the identity
function on $\R^2$.

\begin{proposition}\cite[Section 2]{de2021strong} \label{diffeomorphism}
Let $N\geq 3$ and $\Gamma$ be defined in \eqref{Gamma}
    with $g$ satisfying \eqref{condition_g1} and \eqref{condition_g2}.
  There exist $F=(F_1,\dots,F_N)\in C^{1,1}(\R^N, \R^N)$ and $r_1>0$ such that
  $F\big|_{B_{r_1}}:B_{r_1} \to F(B_{r_1})$  is a
  diffeomorphism of class $C^{1,1}$, 

  \begin{equation}\label{F-propeties}
 F(y',0,0)=(y',g(y'),0) \text{ for any } y' \in \R^{N-1},\quad
 \text{and} \quad F(\tilde \Gamma\cap B_{r_1})=\Gamma\cap F(B_{r_1}),
\end{equation}
with $\tilde\Gamma$ as in \eqref{Gammatilde}.
Furthermore, letting $J_F(y)$ be the Jacobian matrix of
$F$ at $y=(y',y_{N-1},y_N) \!\in~\!\!B_{r_1}$  and
\begin{equation}\label{eq:defA}
  A(y):=|\det J_F(y)|(J_F(y))^{-1}((J_F(y))^{-1})^T,
\end{equation}
the following properties hold:
\begin{enumerate}[\rm i)]\setlength\itemsep{10pt}
\item $J_F$ depends only on the variable $y''=(y',y_{N-1})$ and
\begin{equation}
  \label{J-F}
J_F(y)=J_F(y'')={\mathop{\rm Id}}_N+O(|y''|) \quad
\text{as } |y''|\to 0^+,
\end{equation}
where $\mathop{\rm Id}_{N}$ denotes the identity $N\times N$ matrix and
$O(|y''|)$ denotes a matrix with all entries being $O(|y''|)$ as
$|y''|\to 0^+$;
\item 
$\det{ J_F}(y)=\det{ J_F}(y',y_{N-1}) =1+O(|y'|^2)+O(y_{N-1})$ as
$|y'| \to 0^+$ and $y_{N-1}\to0$;
\item 
$\pd{F_i}{y_N}=\pd{F_N}{y_i}=0$ for
                        any  $i=1,\dots, N-1$ and $\pd{F_N}{y_N}=1$;
                      \item the matrix-valued function $A$ can be written as
                        \begin{equation}
  \label{matrix-A}
A(y)=A(y', y_{N-1})=\left(\begin{array}{c|c}
D(y', y_{N-1}) &0 \\ \hline
0&\det{J_F(y', y_{N-1})}
\end{array}\right),
\end{equation}
with
\begin{equation}\label{matrix-D}
	D(y',y_{N-1})=\left(\begin{array}{c|c}
	\mathop{\rm Id}_{N-2}+O(|y'|^2)+O(y_{N-1}) &O(y_{N-1}) \\ \hline
	O(y_{N-1})&1+O(|y'|^2)+O(y_{N-1})
	\end{array}\right),
\end{equation}
where 
$\mathop{\rm Id}_{N-2}$ denotes the identity $(N-2)\times (N-2)$ 
matrix and $O(y_{N-1})$, respectively $O(|y'|^2)$, denotes
blocks of matrices with all entries being $O(y_{N-1})$ as $y_{N-1}\to 0$, 
respectively $O(|y'|^2)$ as $|y'|\to 0$.
\item 
 $A$ is  symmetric   with coefficients of class $C^{0,1}$ and  
\begin{align}
&\norm{A(y)}_{\mathcal{L}(\R^N,\R^N)} \le 2 \quad \text{for all }y \in B_{r_1},\label{A-oper-norm}\\ 
& \frac12 |z|^2 \le A(y)z \cdot z \le 2 |z|^2 \quad \text{for all } z \in \R^N \text{ and } y \in B_{r_1}. \label{ellipticity}
\end{align} 
\end{enumerate}
\end{proposition}

We observe that
 \begin{equation}\label{eq:N2mat}
   A=\mathop{\rm Id}\nolimits_{2}\quad\text{if }N=2.
\end{equation}
Moreover \eqref{matrix-A}--~\eqref{matrix-D} easily imply that 
\begin{equation}\label{A-O}
A(y)=A(y'')=\mathop{\rm Id}\nolimits_{N}+O(|y''|) \quad \text{as } |y''| \to 0^+.
\end{equation}
Under the same assuptions and with the same notation of
  Proposition \ref{diffeomorphism}, we define
\begin{equation}\label{mu-beta}
	\mu(y):=\frac{A(y)y \cdot y}{|y|^2} \quad \text { and } \quad \beta(y):=\frac{A(y)y }{\mu(y)} 
\quad \text{for any } y\in B_{r_1}\setminus \{0\}.
	\end{equation}

\begin{proposition}\cite[Section 2]{de2021strong} \label{beta-properties}
Under the same assumptions as
  Proposition \ref{diffeomorphism}, let $\mu$ and $\beta$ be as in
  \eqref{mu-beta}.
Then, possibly choosing $r_1$ smaller from the beginning, 
\begin{align}
		&\frac{1}{2}\le\mu(y) \le 2 \quad \text{for any } y\in B_{r_1}\setminus \{0\}, \label{mu-estimates} \\
		&\mu(y)=1 +O(|y|) \quad \text{as } |y| \to 0^+, \label{mu-O}\\
		&\nabla \mu(y)=O(1) \quad  \text{as } |y| \to 0^+.\label{nabla-mu-O}
\end{align} 
Moreover $\beta$ is well-defined and
\begin{align}
\label{eq:27}&\beta(y)=y +O(|y|^2)=O(|y|) \quad \text{as } |y| \to 0^+,\\ \label{beta-estimate}
&J_{\beta}(y)=A(y)+O(|y|)=\mathop{\rm Id}\nolimits_N+O(|y|) \quad \text{as } |y| \to 0^+, \\
&\mathop{\rm div}(\beta)(y)=N+O(|y|) \quad \text{as } |y| \to 0^+.\label{div-beta}
\end{align}
\end{proposition}

We also define $dA(y)zz$, for every $z=(z_1,\dots,z_N) \in \R^N$ and
$y \in B_{r_1}$, as the vector of $\R^N$ with $i$-th component, for
$i=1,\dots,N$, given by
\begin{equation}\label{dA}
(dA(y)zz)_i=\sum_{h,k=1}^N\pd{a_{kh}}{y_i}z_h z_k,
\end{equation}
where we have defined the matrix $A=(a_{k,h})_{k,h=1,\dots, N}$ in \eqref{eq:defA}. 

\begin{remark}\label{tilde-f-regularity}
For any measurable function $f: F(B_{r_1})\to \R$ we set 
\[
  \tilde f:B_{r_1}\to \R,\quad \tilde{f}:=|\det J_F| \,(f \circ F).
\]
Then, in view of i) and ii) in Proposition \ref{diffeomorphism}, the function
$\tilde{f}$ satisfies assumptions \eqref{h1} or \eqref{h2} on
$B_{r_1}$ if and only if $f$ satisfies assumptions \eqref{h1} or
\eqref{h2} on $F(B_{r_1})$.
	
\end{remark} 
It is easy to see that, if $u$ is a solution to \eqref{op}, then the
function $U:=u \circ F$ belongs to
$H^1(B_{r_1}\setminus \tilde \Gamma)$ and is a weak solution of the
problem
\begin{equation}\label{straightenedproblem}\begin{cases}
		-\mathop{\rm{div}}(A\nabla U)=\tilde{f} u, &\text{in } B_{r_1} \setminus \tilde \Gamma,  \\[5pt]
		A\nabla^+ U \cdot \nu^+=A\nabla^- U \cdot \nu^-=0,
		&\text{on } \tilde\Gamma,			
	\end{cases}
\end{equation}
where 
\begin{equation*}
	\nabla^+ U=\nabla \left(U\big|_{B_{r_1}^+}\right),\quad
        \nabla^- U=\nabla \left(U\big|_{B_{r_1}^-}\right),
        \quad \text{and } \nu^-=-\nu^+=(0,\dots, 1).
\end{equation*}
By saying that $U$ a weak solution to
  \eqref{straightenedproblem} we mean that $U \in H^1(B_{r_1} \setminus \tilde\Gamma)$
and
\begin{equation*}
  \int_{B_{r_1}\setminus\tilde \Gamma}(A\nabla U\cdot \nabla \phi -\tilde f U \phi) \, dy=0
\end{equation*}
for all $\phi \in H_{0,\partial B_{r_1}}^1(B_{r_1} \setminus \tilde \Gamma)$.
\section{Traces and embeddings for the space $H^1(B_{r_1}\setminus\tilde\Gamma)$}
\label{sec:trac-embedd-space}
In this section, we present some trace and embedding
inequalities for the space  $H^1(B_{r_1}\setminus \tilde\Gamma)$ which will be
used throughout the paper. 

We define the even reflection operators
\begin{align} \label{even re +}
	&\mathcal{R}^{+}(v)(y',y_{N-1},y_N)=v(y',y_{N-1},|y_N|),\\
	\label{even re -}
	&\mathcal{R}^{-}(v)(y',y_{N-1},y_N)=v(y',x_{N-1},-|y_N|),
\end{align}
and observe that, for all $r>0$,
$\mathcal{R}^{+}:H^1(B_r\setminus\tilde\Gamma)\to H^1(B_r)$ and
$\mathcal{R}^{-}:H^1(B_r\setminus\tilde\Gamma)\to H^1(B_r)$.  We have that
$\mathcal{R}^+(v),\mathcal{R}^-(v) \in L^p(B_r)$ for some
$p \in [1,\infty)$ if and only if $v \in L^p(B_r)$; in such a case we
have that 
\begin{equation} \label{even re p1}
	\norm{\mathcal{R}^{+}(v)}_{L^p(B_r)}^p = 2
	\norm{v}_{L^p(B_r^+)}^p,
	\quad \norm{\mathcal{R}^-(v)}_{L^p(B_r)}^p = 2 \norm{v}_{L^p(B_r^-)}^p,
\end{equation}
and 
\begin{equation}\label{even re p2}
	\norm{v}^p_{L^p(B_r)} 	=\frac{1}{2}\left(\norm{\mathcal{R}^{+}(v)}^p_{L^p(B_r)}+\norm{\mathcal{R}^{-}(v)}_{L^p(B_r)}^p\right).
\end{equation}
Furthermore, for every $v \in H^1(B_r \setminus \tilde\Gamma)$,
\begin{equation} \label{even re nabla}
	\int_{B_r \setminus \tilde \Gamma} |\nabla v|^2 \, dy = \frac{1}{2}\left(\int_{B_r} |\nabla\mathcal{R^+}(v)|^2 \, dy+\int_{B_r} |\nabla\mathcal{R^-}(v)|^2 \, dy\right).
\end{equation}

\begin{proposition} \label{trace r}
	For any $r >0$ there exists a linear continuous trace operator 
	\begin{equation*}
		\gamma_r:H^1(B_r \setminus \tilde\Gamma) \rightarrow L^2(\partial B_r).
              \end{equation*}
              Furthermore $\gamma_r$ is compact.
\end{proposition}
\begin{proof}
	Since $B^+_r$ and $B^-_r$ are Lipschitz domains, there exist two
	linear, continuous and compact trace operators
	$\gamma^+_r:H^1(B^+_r) \rightarrow L^2(\partial
	B^+_r\cap \partial B_r)$ and
	$\gamma^-_r:H^1(B^-_r) \rightarrow L^2(\partial
	B^-_r\cap \partial B_r)$. By setting
	\[
	\gamma_r(v)(y):=
	\begin{cases}
		\gamma^+_r(v)(y), &\text{ if } y_N > 0, \\
		\gamma^-_r(v)(y), &\text{ if } y_N <0,	
	\end{cases}
      \]
 we complete the proof. 
\end{proof}
Letting $\gamma_r$ be the trace operator introduced in Proposition
\ref{trace r}, we observe that 
\begin{equation} \label{even re part} \int_{\partial B_r}
  |\gamma_r(v)|^2 \, dS = \frac{1}{2}\left( \int_{\partial B_r}
    |\gamma_r(\mathcal{R^+}(v))|^2 \, dS+\int_{\partial B_r}
    |\gamma_r(\mathcal{R^-}(v))|^2 \, dS \right)
\end{equation}
 for every $v \in H^1(B_r \setminus \tilde\Gamma)$.
With a slight abuse of notation we will often write  $v$ instead of $\gamma_r(v)$ on $\partial B_r$.

\begin{proposition}\label{hardy}
  If $N \ge 3 $ and $r>0$, then, for any
  $v \in H^1(B_r\setminus{\tilde\Gamma})$,
	\begin{equation} \label{Hardy ineq}
		\Bigr(\frac{N-2}{2}\Bigl)^2\int_{B_r}\frac{v^2}{|x|^2} \, dx \le \int_{B_r \setminus {\tilde\Gamma}} |\nabla v |^2 \, dx +\frac{N-2}{2r}\int_{\partial B_r } v^2\, dS.
	\end{equation}
\end{proposition}
\begin{proof}
	By scaling,  \cite[Theorem 1.1]{wang2003hardy} proves the claim for
	$\mathcal{R}^+(v)$ and $\mathcal{R}^-(v)$. Then we conclude by
	\eqref{even re p2}, 
	\eqref{even re nabla}, and \eqref{even re part}.
\end{proof}

\begin{proposition} \label{sobolev ineq}
Let $N\geq 2$ and $q\geq1$ be such that $1\leq q\leq 2^*=\frac{2N}{N-2}$
	if $N\geq3$ and $1\leq q<\infty$ if $N=2$. Then
	\begin{equation*}
          H^1(B_r\setminus{\tilde\Gamma})\subset  L^q(B_r)\quad\text{for every }r>0
	\end{equation*}    
	and there exists $\mathcal S_{N,q}>0$ (depending only on $N$ and $q$) such that
	\begin{equation} \label{Sobolev ineq} \norm{v}_{L^{q}(B_r)}^2
          \le \mathcal S_{N,q}\,r^\frac{N(2-q)+2q}{q} \left(\int_{B_r \setminus
              {\tilde\Gamma}}|\nabla v|^2 \,
          dx+\frac{1}{r}\int_{\partial B_r}v^2 \, dS\right),
	\end{equation}
	for all $r>0$ and $v \in H^1(B_r\setminus {\tilde\Gamma})$.
\end{proposition}
\begin{proof}
	Since 
	\[
          \left(\int_{B_1}|\nabla v|^2 \,
            dx+\int_{\partial B_1}v^2  \, dS\right)^{\frac{1}{2}}
        \] 
	is an equivalent norm on $H^1(B_1)$, from a scaling argument and
	Sobolev embedding Theorems it follows that, for all $q \in [1,2^*]$ if
	$N\geq3$ and $q\in[1,\infty)$ if $N=2$, there exists $\mathcal S_{N,q}>0$ such
	that, for all $r>0$ and $v \in H^1(B_r)$, 
	\[
	\norm{v}_ {L^{q}(B_r)}^2 \le \mathcal S_{N,q}r^\frac{N(2-q)+2q}{q}
	\left(\int_{B_r}|\nabla v|^2 \, dx+\frac{1}{r}\int_{\partial B_r}v^2
	\, dS\right).
	\]
	Using \eqref{even re p1}, \eqref{even re p2}, \eqref{even re nabla}
	and \eqref{even re part} we complete the proof.
\end{proof}

\begin{proposition}\label{fond ineq}
For any $r>0$, $h\in L^{\frac{N}{2}+ \epsilon }(B_r)$ with
$\epsilon > 0$, and $v \in H^1(B_r\setminus {\tilde\Gamma})$, there holds
\begin{equation}\label{Found ineq}
  \int_{B_r}|h|v^2 \le
  \eta_h(r)\left(\int_{B_r \setminus {\tilde\Gamma}}|\nabla
    v|^2 \, dx+ \frac{1}{r}\int_{\partial B_r}v^2 \,dS\right),
\end{equation}
	where  
	\begin{equation} \label{def eta} \eta_h(r)= \mathcal S_{N,q_\epsilon}
		\norm{h}_{L^{\frac{N}{2}+\epsilon}(B_r)} r^{\frac{4 \epsilon}{N+2
				\epsilon}} \quad \text{and}\quad q_\epsilon:={\frac{2N+4
				\epsilon}{N-2+2 \epsilon}}.
	\end{equation}
\end{proposition}

\begin{proof}
For any $v \in H^1(B_r\setminus {\tilde\Gamma})$
\begin{align*}
  \int_{B_r}|h|v^2 \, dx &\le \norm{h}_{L^{\frac{N}{2}+\epsilon}(B_r)}\left(\int_{B_r}|v|^{q_\epsilon}
                           \, dx \right)^{2/q_\epsilon}\\
                         &\le \mathcal S_{N, q_\epsilon} \norm{h}_{L^{\frac{N}{2}+\epsilon}(B_r)}
                           r^{\frac{4\epsilon}{N+2 \epsilon}}\left(\int_{B_r
                           \setminus
                           {\tilde\Gamma}}|\nabla v|^2\,dx
                           +\frac{1}{r}\int_{\partial B_r}v^2 dS\right)
\end{align*}
	thanks to H\"older inequality and \eqref{Sobolev ineq}.
\end{proof}
\begin{remark} \label{H2 then N/2}
	If $f$ satisfies \eqref{h2}, then $f \in
        L^{\frac{N}{2}+\epsilon }(B_R)$, so that
          Proposition \ref{fond ineq} applies to potentials satisfying either
          \eqref{h1} or
        \eqref{h2}.
\end{remark}
\begin{remark}
  By \eqref{Found ineq}, \eqref{mu-estimates} and \eqref{ellipticity},
  for any $r\in (0,r_1)$, $h \in L^{\frac{N}{2}+ \epsilon }(B_r)$, and
  $v \in H^1(B_r \setminus\tilde\Gamma)$, we have that
 
  \[
    \int_{B_r\setminus\tilde\Gamma}|\nabla v|^2 \, dy \le
    2\int_{B_r\setminus\tilde\Gamma}(A \nabla v \cdot \nabla v-h v^2)
    \, dy +{2\eta_h(r)}\left(\int_{B_r\setminus\tilde\Gamma}|\nabla v|^2
      \, dy+ \frac{2}{r}\int_{\partial B_r}\mu v^2 \,dS\right)
  \]
and therefore, if $\eta_h(r)<{\frac12}$,
\begin{equation} \label{nabla ineq} \int_{B_r\setminus\tilde\Gamma}
  |\nabla v|^2 \, dy \le \frac{2}{1-
    {2\eta_h(r)}}\int_{B_r\setminus\tilde\Gamma}(A \nabla v
  \cdot \nabla v-hv^2) \, dy +
  \frac{{4\eta_h(r)}}{(1-{2\eta_h(r)})r}
  \int_{\partial B_r}\mu v^2 \,dS.
	\end{equation}
\end{remark}
\section{Approximating problems}\label{sec:appr-probl}
In this section we construct a sequence of problems in smooth sets
approximating the straightened cracked domain.
We define, for any $n \in \mathbb{N} \setminus \{0\}$,  

\[
  g_n:\R\rightarrow \R, \quad g_n(t):=nt^4
\]
and, for any $r \in (0,r_1]$,
\begin{equation*}
\Omega_{n,r}:=\{(y',y_{N-1},y_N)\in B_r:y_{N-1}< g_n(y_N)\}
\end{equation*}
and
\begin{equation*}
	 \Gamma_{n,r}:=\{(y',y_{N-1},y_N) \in B_r:y_{N-1}=g_n(y_N)\}=
	\partial\Omega_{n,r}\cap B_r. 
\end{equation*}
It is clear that, for every $y \in B_r \setminus \tilde\Gamma$, there exists
a $\bar n \in \mathbb{N} \setminus \{0\}$  such that $y \in
\Omega_{n,r}$ for all $n\geq\bar n$.   Moreover $\Omega_{n,r} \cap \tilde \Gamma = \emptyset$ 
for any $r \in (0,r_1]$ and $ n \in \mathbb N \setminus\{0\}$.
We also note that  $\Omega_{n,r}$ is a Lipschitz domain and
$\Gamma_{n,r}$ is a $C^2$-smooth portion of its boundary.
\begin{proposition} \label{negativity}
  Let $\nu(y)$ be the outward
  normal vector to $\partial \Omega_{n,r_1}$ in
  $y$. Then
  \begin{align}
    \label{eq:24}{y \cdot \nu(y) \le 0}\quad &\text{for all $y\in
    \Gamma_{n,r_1}$},\\
    \label{eq:25}A(y)y \cdot \nu(y) \le 0\quad &\text{for all $y\in
    \Gamma_{n,r_1}$}.
  \end{align}
\end{proposition}
\begin{proof}
	As a first step we notice that
	\begin{equation}\label{eq:4}
          g_n(t)-\frac{1}{3}tg'_n(t)=nt^4-\frac43nt^4=-\frac{1}{3}nt^4\le
          0, \quad  g_n(t)-tg'_n(t)\le 0
        \end{equation}
	and that 
\[
  \nu(y)=\frac{(0,1,-g_n'(y_N))}{\sqrt{1+(g_n'(y_N))^2}}
  \quad{\text{for all }y\in \Gamma_{n,r_1}}.
\]
Then,  for all $y\in  \Gamma_{n,r_1}$,
\[
  \nu(y) \cdot y=\frac{(0,1,-g_n'(y_N))}{\sqrt{1+(g_n'(y_N))^2}}\cdot
  (y',g_n(y_N),y_N)=\frac{g_n(y_N)- y_Ng_n'(y_N)
  }{\sqrt{1+(g_n'(y_N))^2}} \le 0
\]
due to \eqref{eq:4}. We have then proved  \eqref{eq:24} (and
\eqref{eq:25} in the case $N=2$ in view of \eqref{eq:N2mat}).

If $N\ge 3$, possibly choosing $r_1$ smaller in
  Proposition \ref{diffeomorphism},  for all $y\in
  \Gamma_{n,r_1}$ we have that
\begin{align*}
\sqrt{1+(g_n'(y_N))^2} A(y)y\cdot \nu(y)
    &=g_n(y_N)(1+{O(|y'|)}+O(y_{N-1}))-\det{ J_F(y)}\,y_Ng_n'(y_N)\\
   &\le\frac{3}{2}g_n(y_N)-\frac{1}{2}y_Ng_n'(y_N)=
     \frac{3}{2}(g_n(y_N)-\frac{1}{3}y_Ng_n'(y_N)),
\end{align*}
thanks to ii) in Proposition \ref{diffeomorphism}, \eqref{matrix-A}
and \eqref{matrix-D}.
Then, by \eqref{eq:4} we finally obtain
  \eqref{eq:25} also for $N\geq3$.
\end{proof}

\noindent Let 
\begin{equation*}
  \R^N_+:=\{y=(y',y_{N-1},y_{N}) \in \R^N: y_N >0\}
  \text{ and }  \R^N_-:=\{y=(y',y_{N-1},y_{N}) \in \R^N: y_N < 0\}.
\end{equation*}
For any $r \in (0,r_1]$ and $n \in \mathbb N \setminus\{0\}$ let 
\begin{equation}\label{Omega Gamma}
	\Omega_{n,r}^+:= \Omega_{n,r} \cap B^+_r, \quad
	\Omega_{n,r}^- :=\Omega_{n,r} \cap B^-_r, \quad S_{n,r}:=\partial \Omega_{n,r} \cap \partial B_r.	
\end{equation}
For all $n\in \mathbb{N}\setminus\{0\}$ we also define
\begin{align*}
	&  K_{n,r_1}^+:=
	\{y=(y',y_{N-1},y_N)\in \R^N_+:\text{either }y_{N-1}<
	g_n(y_N)\text{ or }|y|>r_1\},\\
	& K_{n,r_1}^-:=
	\{y=(y',y_{N-1},y_N)\in \R^N_-:\text{either }y_{N-1}<
g_n(y_N)\text{ or }|y|>r_1\}.
\end{align*}
Since $\Omega_{n,r}$ is a Lipschitz domain, for any $r \in (0,r_1]$  and $n \in \mathbb N \setminus\{0\}$ there exists a trace operator
\begin{equation*}
	\gamma_{n,r}:H^1(\Omega_{n,r}) \rightarrow L^2(\partial \Omega_{n,r}).
\end{equation*}
We define
\begin{equation}\label{H10}
	H^1_{0,S_{n,r}}(\Omega_{n,r}):=\{u \in H^1(\Omega_{n,r}):\gamma_{n,r}(u)=0 \text{ on } S_{n,r} \}.
\end{equation}
The following proposition provides an extension operator from
$H^1_{0,S_{n,r}}(\Omega_{n,r})$ to $H^1(B_{r_1}\setminus\tilde\Gamma)$ with an
operator norm bounded uniformly with respect to $n$.

\begin{proposition} \label{pon ex}
  For any $r \in (0,r_1)$  and $n \in \mathbb N \setminus\{0\}$ there exists an extension operator 
	\begin{equation}
		\xi^0_{n,r}:H^1_{0,S_{n,r}}(\Omega_{n,r})\rightarrow H^1(B_{r_1}\setminus\tilde\Gamma)	
	\end{equation} such that, for any $\phi \in
	H^1_{0,S_{n,r}}(\Omega_{n,r})$, 
	\begin{equation}\label{pon ex properties}
		\xi^0_{n,r}(\phi)\big|_{\Omega_{n,r}}=\phi,\quad   
		\xi^0_{n,r}(\phi)=0   \text{ on }
                \Omega_{n,r_1}\setminus\Omega_{n,r},
                \quad \xi^0_{n,r}(\phi)\in H^1_{0,\partial B_{r_1}}(B_{r_1} \setminus \tilde\Gamma), 
	\end{equation}
	and
	\begin{equation}\label{pon ex ineq}
		\norm{\xi^0_{n,r}(\phi)}_{H^1(B_{r_1}\setminus\tilde\Gamma)} \le c_0
		\norm{\phi}_{H^1(\Omega_{n,r})}
		=c_0\bigg(\int_{\Omega_{n,r}}\big(\phi^2+|\nabla \phi|^2\big)\,dy\bigg)^{1/2}
		,	
	\end{equation}
	where $c_0 > 0$  is independent of $n$, $r$, and $\phi$. 
\end{proposition}	

\begin{proof}
	It is well known that, since $K_{n,r_1}^+$ and $K_{n,r_1}^-$ are
	uniformly Lipschitz domains, there exist continuous extension
	operators $\xi_{n}^+:H^1(K_{n,r_1}^+)\rightarrow H^1(\R^N_+)$ and
	$\xi_{n}^-:H^1(K_{n,r_1}^-)\rightarrow H^1(\R^N_-)$, see
	\cite{stein2016singular}, \cite{calderon1961lebesgue} and
	\cite{leoni2017first}.  Furthermore, since the Lipschitz constants of
	the parameterization of 
	$\partial K_{n,r_1}^+$ and $\partial K_{n,r_1}^-$ are bounded uniformly
	with respect to
	$n$, there exists a constant $C>0$, which does not depend on $n$,
	such that
	\begin{equation} \label{pon ex ineqs}
          \norm{\xi_{n}^+(v)}_{H^1(\R^N_+)}\le
         C\norm{v}_{H^1(K_{n,r_1}^+)} \quad \text{and}
          \quad\norm{\xi_{n}^-(w)}_{H^1(\R^N_-)}\le
          C\norm{w}_{H^1(K_{n,r_1}^-)}
	\end{equation}
	for all $v\in H^1(K_{n,r_1}^+)$ and $w\in H^1(K_{n, r_1}^-)$.
	
	If $\phi \in H^1_{0,S_{n,r}}(\Omega_{n,r})$ then the trivial extension
	$\bar \phi_+$ of $\phi\big|_{\Omega_{n,r}^+}$ to $K_{n,r_1}^+$ belongs to
	$H^1(K_{n,r_1}^+)$ and the trivial extension $\bar \phi_-$ of
	$\phi\big|_{\Omega_{n,r}^-}$ to $K_{n,r_1}^-$ belongs to $H^1(K_{n,r_1}^-)$.
	Then we define 
	\begin{equation*}
		\xi^0_{n,r}(\phi)(y):=
		\begin{cases}
			\xi_{n}^+(\bar \phi_+)(y), &\text{ if } y\in B_{r_1}^+, \\
			\xi_{n}^-(\bar \phi_-)(y), &\text{ if }  y\in B_{r_1}^-,
		\end{cases}
	\end{equation*}
which belongs to
          $H^1(B_{r_1}\setminus\tilde\Gamma)$ and satisfies \eqref{pon ex ineq} in view of \eqref{pon ex ineqs}. Furthermore \eqref{pon ex properties} follows directly from
	the definition of $\xi^0_{n,r}$.
\end{proof}

  The  following proposition establishes a Poincaré type  inequality for
$H^1_{0,S_{n,r}}(\Omega_{n,r})$-functions, with a constant independent
of $n$.
\begin{proposition} \label{Pon Prop} For any $r\in (0,r_1]$,
  $n\in{\mathbb N}\setminus\{0\}$, and
  $\phi \in H^1_{0,S_{n,r}}(\Omega_{n,r})$
	\begin{equation}\label{Pon ineq}
          \int_{\Omega_{n,r}}\phi^2dy \le \frac{r^2}{N-1}\int_{\Omega_{n,r}}|\nabla\phi|^2\,dy 
	\end{equation}
	and
	\begin{equation}\label{equivalent norm}
          \norm{\phi}_{H^1_{0,S_{n,r}}(\Omega_{n,r})}:=\left(\int_{\Omega_{n,r}}|\nabla\phi|^2\,dy \right)^\frac{1}{2}
	\end{equation} is an equivalent norm on $H^1_{0,S_{n,r}}(\Omega_{n,r})$.
      \end{proposition}

\begin{proof}
  For any $\phi \in C^\infty(\overline{\Omega}_{n,r})$ such that
  $\phi=0$ in a neighbourhood of $S_{n,r}$ we have that
	\[
	\mathop{\rm div}(\phi^2 y)= 2\phi \nabla\phi \cdot y+N\phi^2\]
	so that 
	\[
	N\int_{\Omega_{n,r}}\phi^2\,dy=-2\int_{\Omega_{n,r}}\phi \nabla\phi\cdot y 
	\,dy+\int_{ \Gamma_{n,r}}\phi^2 y\cdot \nu \,dS\le
	\int_{\Omega_{n,r}}\phi^2\, dy+r^2\int_{\Omega_{n,r}}|\nabla
	\phi|^2\,dy,
	\]
	since $y\cdot \nu \le 0 $ on $ \Gamma_{n,r}$ by\eqref{eq:24}. Then we may  conclude that 
	\[\int_{\Omega_{n,r}}\phi^2\,dy \le \frac{r^2}{N-1}\int_{\Omega_{n,r}}|\nabla\phi|^2\,dy,\]
	for all $\phi \in C^\infty(\overline{\Omega}_{n,r})$ such that $\phi=0$ in a neighbourhood of 
	$S_{n,r}$. Since $\Omega_{n,r}$ is a Lipschitz domain,  \eqref{Pon ineq} holds for any $\phi \in  H^1_{0,S_{n,r}}(\Omega_{n,r})$ by \cite[Theorem 3.1]{bernard2011density}. The second claim is now obvious.
\end{proof}
From now on we consider on $H^1_{0,S_{n,r}}(\Omega_{n,r})$  the norm
$\norm{\,\cdot\, }_{H^1_{0,S_{n,r}}}$ defined in \eqref{equivalent norm}.

\begin{proposition}\label{Omega found}
	Let $r \in (0,r_1)$, $n\in{\mathbb N}\setminus\{0\}$,   $h \in
	L^{\frac{N}{2}+\epsilon}(B_r)$ with $\epsilon > 0$, and $q_\epsilon$
	be as in \eqref{def eta}. Then, for any $\phi \in H^1_{0,S_{n,r}}(\Omega_{n,r})$,
	\begin{equation} \label{Omega found ineq}
          \int_{\Omega_{n,r}}|h|\phi^2\, dy\le
          c_0^{2}\frac{N-1+r_1^2}{N-1} \mathcal{S}_{N,q_\epsilon}r_1^{\frac{4 \epsilon}{N+2 \epsilon}}
          \norm{h}_{L^{\frac{N}{2}+\epsilon}(B_r)}\int_{\Omega_{n,r}}|\nabla\phi|^2\,
          dy.
	\end{equation}
\end{proposition}

\begin{proof}
	We have, for every  $\phi \in H^1_{0,S_{n,r}}(\Omega_{n,r})$, 
	\begin{align*}
          \int_{\Omega_{n,r}}|h|\phi^2\, dy
          &\le  \int_{B_r}|h||\xi^0_{n,r}(\phi)|^2\, dy\le 
            \norm{h}_{L^{\frac{N}{2}+\epsilon}(B_r)}\left(\int_{B_{r_1}}
            |\xi^0_{n,r}(\phi)|^{q_\epsilon}\, dy \right)^{\!\!\frac{2}{q_\epsilon}}\\
          &\le \mathcal S_{N,q_\epsilon}
            r_1^{\frac{4
            \epsilon}{N+2
            \epsilon}}\norm{h}_{L^{\frac{N}{2}+\epsilon}(B_r)}\int_{B_{r_1}\setminus
            \tilde\Gamma}
            |\nabla\xi^0_{n,r}(\phi)|^2\, dy  \\
          &\le c_0^{2}\frac{N-1+{r^2}}{N-1}
            \mathcal S_{N, q_\epsilon}r_1^{\frac{4 \epsilon}{N+2 \epsilon}}
            \norm{h}_{L^{\frac{N}{2}+\epsilon}(B_r)}\int_{\Omega_{n,r}}|\nabla\phi|^2\, dy,
	\end{align*}
	thanks to H\"older's inequality, \eqref{Sobolev ineq}, Proposition
	\ref{pon ex}, and Proposition \ref{Pon Prop}.
\end{proof}

Hereafter we fix a potential $f$ satisfying either
  \eqref{h1} or \eqref{h2} and define $\tilde f:=|\det J_F| \,(f \circ
  F)$ as in Remark \ref{tilde-f-regularity}.
Thanks to Remark \ref{tilde-f-regularity} 
we have that $\tilde f$ satisfies either
\eqref{h1} or \eqref{h2} as well.
 If $f$ (and consequently $\tilde f$) satisfies \eqref{h2},
  we define
  \begin{equation}\label{fn2}
    f_n(y)=
    \begin{cases}
      n,&\text{if }\tilde f(y)>n,\\
      \tilde f(y),&\text{if }|\tilde f(y)|\leq n,\\
      -n,&\text{if }\tilde f(y)<-n, 
    \end{cases}
  \end{equation}
  so that
  \begin{equation}\label{eq:5}
    f_n\in L^\infty(B_{r_1}) \quad\text{and}\quad
    |f_n|\leq |\tilde f|\text{ a.e. in $B_{r_1}$}\quad\text{for all }n\in {\mathbb N}\setminus\{0\} 
  \end{equation}
  and
    \begin{equation}\label{eq:26}
      f_n\to\tilde f\text{ a.e. in $B_{r_1}$}.
    \end{equation}
  If $f$ satisfies \eqref{h1},  we just let 
\begin{equation}\label{fn1}
f_n:=\tilde{f} \quad \text{for any } n \in \mathbb{N}.
\end{equation} 
We observe that
  \begin{equation}\label{eq:3}
    f_n\to\tilde f\quad\text{in
    }L^{\frac{N}{2}+\epsilon}(B_{r_1})\quad\text{as }n\to\infty
  \end{equation}
  as a consequence of \eqref{eq:5}, \eqref{eq:26} and the Dominated
  Convergence Theorem if  assumption \eqref{h2} holds and $f_n$ is
  defined in \eqref{fn2}, in view of Remark \ref{H2 then N/2}; on
  the other hand \eqref{eq:3} is obvious if  assumption \eqref{h1} holds and $f_n$ is
  defined  in \eqref{fn1}.

Since under both assumptions \eqref{h1} and
  \eqref{h2}  we have that $\tilde f\in
  L^{\frac{N}{2}+\epsilon}(B_{r_1})$ (see Remark \ref{H2 then N/2}), by the absolute
        continuity of the Lebesgue integral we can choose 
  $r_0 \in (0,\min\{1,r_1\})$ such that
    \begin{equation} \label{beta1}
	\eta_{\tilde
      f}(r_0)<\frac{1}{2}\quad\text{and}\quad	c_0^{2}\frac{N-1+{r_1^2}}{N-1}
                \mathcal S_{N,q_\epsilon}{r_1^{\frac{4
                      \epsilon}{N+2 \epsilon}}}
                \|{\tilde f}\|_{L^{\frac{N}{2}+\epsilon}(B_{r_0})} < \frac14,
	\end{equation}
	where $q_\epsilon$ and $\eta_{\tilde f}$ are defined in \eqref{def eta}.

Let $U= u\circ F$, where $u$ is a fixed  weak solution to \eqref{op}
and $F$ is the diffeomorphism introduced in Section
\ref{sec:boundary}, so that $U$ weakly solves
\eqref{straightenedproblem}. For any $n \in\mathbb{ N}\setminus\{0\}$, we
      consider the following sequence of approximating problems, with
      potentials $f_n$ defined in \eqref{fn2}--\eqref{fn1}:
\begin{equation} \label{aprox prom}
	\begin{cases}
		-\mathop{\rm{div}}(A\nabla U_n)= f_n U_n, &\text{ in } \Omega_{n,r_0},  \\
	A\nabla U_n \cdot \nu=0,  &\text{ on }  \Gamma_{n,r_0},	\\
		\gamma_{n,r_0}(U_n)=\gamma_{n, r_0}(U), &\text{ on }  S_{n,r_0},
	\end{cases}
\end{equation}
with $r_0$  as in \eqref{beta1}.
A weak solution to problem \eqref{aprox prom} is a function $ U_n \in
H^1(\Omega_{n,r_0})$ such that
$U_n-U\in H^1_{0,S_{n,r_0}}(\Omega_{n,r_0})$ and 
\begin{equation*}
	\int_{\Omega_{n,r_0}}(A\nabla U_n \cdot\nabla \phi-f_nU_n \phi)\, dy =0	
\end{equation*}
for all $\phi \in H^1_{0,S_{n,r_0}}(\Omega_{n,r_0})$.
If $U_n$ weakly solves \eqref{aprox prom}, then $ W_n:=U-U_n\in
H^1_{0,S_{n,r_0}}(\Omega_{n,r_0})$ and 
\begin{equation} \label{ape}
	\int_{\Omega_{n,r_0}}(A\nabla  W_n \cdot\nabla \phi - f_n  W_n \phi)\, dy =\int_{\Omega_{n,r_0}}(A\nabla U \cdot\nabla \phi -f_n U \phi )\, dy
\end{equation}
for any  $\phi \in H^1_{0,S_{n,r_0}}(\Omega_{n,r_0})$.

For every $n\in{\mathbb N}\setminus\{0\}$, let us consider the bilinear form 
\begin{equation}\label{eq:Bn}
	B_n:H^1_{0,S_{n,r_0}}(\Omega_{n,r_0})\times
	H^1_{0,S_{n,r_0}}(\Omega_{n,r_0})\rightarrow \mathbb{R}, \quad
	B_n(v,\phi):=\int_{\Omega_{n,r_0}}(A\nabla v\cdot\nabla \phi - f_nv \phi)\,
	dy,
\end{equation}
and the functional
\begin{equation}\label{eq:Ln}
	L_n: H^1_{0,S_{n,r_0}}(\Omega_{n,r_0})\to\R,\quad 
	L_n(\phi):=\int_{\Omega_{n,r_0}}(A\nabla U \cdot\nabla \phi - f_n U \phi )\,
	dy.
\end{equation}

\begin{proposition}  \label{lmh}
	The bilinear form $B_n$ defined in \eqref{eq:Bn} is continuous and
	coercive; more precisely 
	\begin{equation}\label{coer}
		B_n(\phi,\phi)\ge \frac14  \norm{\phi}^2_{H^1_{0,S_{n,r_0}}(\Omega_{n,r_0})}\quad\text{for all $\phi\in H^1_{0,S_{n,r_0}}(\Omega_{n,r_0})$}.
	\end{equation}
	Furthermore the functional $L_n$ defined in \eqref{eq:Ln} 
	belongs to $(H^1_{0,S_{n,r_0}}(\Omega_{n,r_0}))^*$ and there exists a
	constant  $\ell > 0$ independent of $n$ such that
	\begin{equation}\label{cont}
		|L_n(\phi)|\le
		\ell\norm{\phi}_{H^1_{0,S_{n,r_0}}(\Omega_{n,r_0})}\quad\text{for all
			$\phi \in H^1_{0,S_{n,r}}(\Omega_{n,r_0})$}.	
	\end{equation}
\end{proposition}

\begin{proof}
  The continuity of $B_n$ and \eqref{coer} easily follow from
  \eqref{ellipticity},\eqref{eq:5}, \eqref{Omega found ineq} and \eqref{beta1}.
  Thanks to H\"older's inequality,
 \eqref{eq:5},
 \eqref{A-oper-norm}, \eqref{Found ineq}, \eqref{Omega
    found ineq} and \eqref{beta1}
	\begin{align*}
          &|L_n(\phi)|\le{2}\norm{\nabla
            {U}}_{L^2(\Omega_{n,r_0})}
            \norm{\phi}_{H^1_{0,S_{n,r_0}}(\Omega_{n,r_0})} 
            +\left(\int_{B_{r_0}}|{\tilde f}|{U}^2\, dx
            \right)^{\!\!\frac{1}{2}}\left(\int_{\Omega_{n,r_0}}|{\tilde f}|\phi^2 \, dx\right)^{\!\!\frac{1}{2}}\\
          &\le \left({2}\norm{\nabla
            {U}}_{L^2(B_{r_0}\setminus\tilde\Gamma)}+
            {\frac12}\sqrt{\eta_{\tilde f}(r_0)}\left(\int_{B_{r_0}
            \setminus \tilde\Gamma}
            |\nabla {U}|^2 \, dx+ \frac{1}{r_0}
            \int_{\partial B_{r_0}}{U}^2
            \,dS\right)^{\!\!\frac{1}{2}}\right)
            \norm{\phi}_{H^1_{0,S_{n,r_0}}(\Omega_{n,r_0})},
	\end{align*}
thus implying \eqref{cont}.
\end{proof}

\begin{corollary}\label{existance}
  Let $u$ be a weak solution to \eqref{op} and 
 $U= u\circ F$.  Let either \eqref{h1} hold
	and $\{f_n\}$ be as in \eqref{fn1}, or \eqref{h2} hold and $\{f_n\}$
	be as in \eqref{fn2}. Let $r_0$ be as in \eqref{beta1}
	and $\ell$ be as in Proposition \ref{lmh}. Then, for any
	$n \in {\mathbb N}\setminus\{0\}$, there exists a solution
	$ W_n \in H^1_{0,S_{n,r_0}}(\Omega_{n,r_0})$ of \eqref{ape} such that
	\begin{equation} \label{existance estimates}
		\norm{ W_n}_{H^1_{0,S_{n,r_0}}(\Omega_{n,r_0} )}\le 4\ell.
	\end{equation}
\end{corollary}

\begin{proof}
	The existence of a solution $W_n$ of \eqref{ape}
        follows from the
	Lax-Milgram Theorem, taking into account Proposition
	\ref{lmh}. Estimate  \eqref{existance estimates} follows from
	\eqref{coer} and  \eqref{cont} with $\phi= W_n$.
\end{proof}

We are now in position to prove the main result of this section.

\begin{theorem} \label{aprox theor} Suppose that $f$ satisfies
	either   \eqref{h1} or \eqref{h2}, $u$ is a weak solution of
	\eqref{op}, and $U=u \circ F$ with $F$ as in Section
\ref{sec:boundary}. Let
	$\{f_n\}_{n \in \mathbb{N}}$ satisfies \eqref{fn1} under hypothesis
	\eqref{h1} or \eqref{fn2} under  hypothesis
	\eqref{h2}.  Let $r_0 \in (0,r_1)$ be as \eqref{beta1}.
	Then there exists 
	$\{U_n\}_{n \in {\mathbb N}\setminus\{0\}} \subset H^1(B_{r_0} \setminus \tilde  \Gamma)$
	such that $U_n$ weakly solves \eqref{aprox prom} for
	any $n\in{\mathbb N\setminus\{0\}}$ and $U_n \rightarrow U$ in
	$H^1(B_{r_0} \setminus \tilde \Gamma)$ as $n\to\infty$. Furthermore
	$U_n \in H^2(\Omega_{n,r})$ for any $r \in(0,r_0)$ and
	$n \in{\mathbb N}\setminus\{0\}$.
\end{theorem}
\begin{proof}
  Let $r_0 \in (0,r_1)$ be as in \eqref{beta1}. For any
  $n \in{\mathbb N}\setminus\{0\}$, let
  $W_n \in H^1_{0,S_{n,r_0}}(\Omega_{n,r_0})$ be the solution to
  \eqref{ape} given by Corollary \ref{existance}. Then $U- W_n$ weakly
  solves problem \eqref{aprox prom} and we define
  $U_n:= U-\xi^0_{n,r_0}( W_n)$, with $\xi^0_{n,r_0}$
    being the extension operator introduced in Proposition \ref{pon ex}.  We observe that
  $U_n \in H^1(B_{r_0}\setminus \tilde \Gamma)$. To prove that $U_n$
  converges to $U$ in $H^1(B_{r_0} \setminus \tilde \Gamma)$ as
  $n\to \infty$, we notice that
\[\norm{U-U_n}^2_{H^1(B_{r_0}\setminus \tilde \Gamma)} \le c^2_0 \norm{ W_n}_{H^1(\Omega_{n,r_0})}^2 
          \le 4\,c^2_0\,\frac{
            N-1+r_0^2}{N-1}\int_{\Omega_{n,r_0}}(A\nabla W_n\cdot
          \nabla W_n- f_n W_n^2)\, dy,
        \]
 by Proposition \ref{pon
          ex}, \eqref{Pon ineq}, and \eqref{coer}.
        Therefore it is enough to prove that
	\begin{equation}\label{lim=0}
		\lim_{n \to \infty}\int_{\Omega_{n,r_0}}(A\nabla  W_n\cdot \nabla  W_n- f_n W_n^2)\, dy=0.
	\end{equation}
	Let
       
        \begin{equation}\label{eq:30}
         O_n:=(B_{r_1}\setminus \tilde \Gamma) \setminus
          \Omega_{n,r_1}
        \end{equation}
for any  $n\in{\mathbb N}\setminus\{0\}$. Since $W_n \in H^1_{0,S_{n,r_0}}(\Omega_{n,r_0})$
	solves \eqref{ape} and $U$ is a solution to \eqref{straightenedproblem}, by
	H\"older's inequality, \eqref{A-oper-norm} and Proposition \ref{pon ex} we have that
	\begin{align*}
          \bigg|&\int_{\Omega_{n,r_0}}(A\nabla  W_n\cdot \nabla  W_n- f_n W_n^2)\,
                  dy\bigg|=\left|\int_{\Omega_{n,r_1}}(A\nabla U\cdot
                  \nabla( \xi^0_{n,r_0}( W_n))
                  - f_n U \,\xi^0_{n,r_0}( W_n))\, dy\right|\\
		&=\bigg|\int_{B_{r_1}\setminus \tilde \Gamma}(A\nabla U\cdot \nabla(
           \xi^0_{n,r_0}( W_n))- f_n U \,\xi^0_{n,r_0}( W_n))\,dy\\
		&\qquad\qquad-\int_{O_n}(A\nabla U \cdot\nabla(
           \xi^0_{n,r_0}( W_n))-
           f_nU \, \xi^0_{n,r_0}( W_n))\,dy\bigg|\\
		&=\bigg|\int_{B_{r_1}\setminus \tilde \Gamma}(A\nabla U\cdot \nabla(
           \xi^0_{n,r_0}( W_n))- \tilde f U \,\xi^0_{n,r_0}( W_n))\,dy+\int_{B_{r_1}\setminus \tilde \Gamma}(\tilde f-f_n) U \,\xi^0_{n,r_0}( W_n))\,dy\\
		&\qquad\qquad-\int_{O_n}(A\nabla U \cdot\nabla(
           \xi^0_{n,r_0}( W_n))-
           f_nU \, \xi^0_{n,r_0}( W_n))\,dy\bigg|\\
		&\leq\left|\int_{O_n}(A\nabla U \cdot\nabla(
           \xi^0_{n,r_0}( W_n))-
           f_nU\, \xi^0_{n,r_0}( W_n))\,dy \right|+
\left|\int_{B_{r_1}\setminus \tilde \Gamma}(\tilde f-f_n) U \,\xi^0_{n,r_0}( W_n)\,dy\right|
                          \\
		&\le 2\norm{\nabla U}_{L^2(O_n)}
           \norm{\nabla \xi^0_{n,r_0}( W_n)}_{L^2(B_{r_1}\setminus \tilde \Gamma)}+
           \| f_n\|_{L^{\frac{N}{2}+\epsilon}(O_n)}\norm{U}_{L^{q_\epsilon}(O_n)}
           \norm{\xi^0_{n,r_0}( W_n)}_{L^{q_\epsilon}(B_{r_1})}\\
&\quad+\|\tilde f- f_n\|_{L^{\frac{N}{2}+\epsilon}(B_{r_1})}\norm{U}_{L^{q_\epsilon}(B_{r_1})}
           \norm{\xi^0_{n,r_0}( W_n)}_{L^{q_\epsilon}(B_{r_1})}\\
		&\le4 c_0\ell \frac{\sqrt{N-1+r^2_0}}{\sqrt{N-1}}
           \Big(2
           \norm{\nabla U}_{L^2(O_n)}
           +\sqrt{\mathcal S_{N,
           q_\epsilon}}r_1^{\frac{2\epsilon}{N+2\epsilon}}
           \|{\tilde f}\|_{L^{\frac{N}{2}+\epsilon}(O_n)}\norm{U}_{L^{q_\epsilon}(O_n)}\\
                          &\hskip6cm
          +\sqrt{\mathcal S_{N,
           q_\epsilon}}r_1^{\frac{2\epsilon}{N+2\epsilon}}\|\tilde f- f_n\|_{L^{\frac{N}{2}+\epsilon}(B_{r_1})}\norm{U}_{L^{q_\epsilon}(B_{r_1})}\Big),
	\end{align*}
	where $q_\epsilon$ is defined in \eqref{def eta} and  we have
        used \eqref{eq:5},
        \eqref{Sobolev ineq}, \eqref{pon ex ineq}, \eqref{Pon ineq},  and \eqref{existance estimates} in the last inequality. We observe that 
	\[
          \lim_{n \to \infty}|O_n|=0,
        \] 
	where $|O_n|$ is the $N$-dimensional Lebesgue
        measure of $O_n$. Then, since $\nabla U\in
        L^2(B_{r_1}\setminus\tilde\Gamma)$, $U \in
	L^{q_\epsilon}(B_{r_1})$ by Proposition \ref{sobolev ineq}, and
$\tilde f\in L^{\frac{N}{2}+
          \epsilon}(B_{r_1})$, \eqref{lim=0} follows by the absolute
        continuity of the integral
 and convergence \eqref{eq:3}.

We observe that $f_n U_n\in L^2(\Omega_{n,r_0})$. Indeed,
   under assumption \eqref{h1}, by Remark
   \ref{tilde-f-regularity} we have that
   $\tilde f \in W^{1,\frac{N}{2}+\epsilon}(B_{r_1}\setminus \tilde\Gamma)$
   and then, by Sobolev embeddings and H\"older's inequality, we easily
   obtain that $f_n U_n=\tilde f U_n\in L^2(\Omega_{n,r_0})$.
Under assumption \eqref{h2}, $f_n$ is defined in \eqref{fn2} and
$f_n\in L^\infty(B_{r_1})$, hence $f_n U_n\in L^2(\Omega_{n,r_0})$. 

Since $\Gamma_{n,r_0}$ is $C^\infty$-smooth and
$f_nU_n \in L^2(\Omega_{n,r_0})$, by classical elliptic regularity
theory, see e.g. \cite[Theorem 2.2.2.5]{grisvard}, we deduce that
$U_n \in H^2(\Omega_{n,r})$ for any $r \in(0,r_0)$. The proof is
thereby complete.
\end{proof}

\section{The Almgren type frequency function}\label{sec:almgr-type-freq}
Let $u\in H^1(B_R\setminus\Gamma)$ be a non-trivial weak solution to
\eqref{op} and $U=u \circ F\in H^1(B_{r_1} \setminus \tilde \Gamma)$
be the corresponding solution to \eqref{straightenedproblem}.  Let
$r_0 \in (0,\min\{1,r_1\})$ be as in \eqref{beta1}.  For any
$r \in (0,r_0]$, we define
\begin{equation}\label{H}
	H(r):=\frac{1}{r^{N-1}}\int_{\partial B_r}\mu \,U^2\, dS,
\end{equation} 
where $\mu$ is  the  function introduced in \eqref{mu-beta}, and 
\begin{equation}\label{E}
  E(r):=\frac{1}{r^{N-2}}\int_{B_r\setminus \tilde  \Gamma}
  (A\nabla U \cdot \nabla U-\tilde f\,U^2)\,dy.
\end{equation}

\begin{proposition}  \label{H positive}
	If $r \in (0,r_0]$ then $H(r) > 0$.
\end{proposition}

\begin{proof}
  We suppose by contradiction that there exists $r\in (0,r_0]$ such
  that $H(r) = 0$. By \eqref{mu-estimates}, it follows that $U$ weakly
  solves \eqref{straightenedproblem} with the extra condition $U=0$ on
  $\partial B_{r}$. Then by \eqref{nabla ineq} we obtain that $U=0$ on
  $B_r$.  By classical unique continuation principles for elliptic
  equations, see e.g. \cite{garofalo1986monotonicity}, we conclude
  that $u=0$ on $B_{R}$, which is a contradiction.
\end{proof}

\begin{proposition}\label{p:derH} 
	We have that $H \in W^{1,1}_{\rm loc}((0,r_0])$ and 
	\begin{align}\label{H'}
		H'(r)&=\frac{1}{r^{N-1}}\left(2\int_{\partial B_r}\mu U \pd{U}{\nu} \, dS +\int_{\partial B_r} U^2 \nabla \mu\cdot \nu \,dS\right)\\
		&=\frac{2}{r^{N-1}}\int_{\partial B_r}\mu U \pd{U}{\nu} \, dS+H(r)O(1) \quad \text{as } r \to 0^+,\notag
	\end{align}
	in a distributional sense and for a.e. $r \in (0,r_0)$.
\end{proposition}
\begin{remark}\label{rem:gradsense}
	To explain in what sense the term $\frac{\partial U}{\partial \nu}$
	 in \eqref{H'} is meant, we observe that, if $\nabla U$ is the
	distributional gradient of $U$ in $B_{r_1}\setminus \tilde \Gamma$, then
	$\nabla U\in L^2(B_{r_1},\R^N)$ and 
	$\frac{\partial U}{\partial \nu}:=\nabla U\cdot \frac{y}{|y|}\in
	L^2(B_{r_1})$. By the Coarea Formula it follows that $\nabla U\in L^2(\partial
	B_r,\R^N)$ and
	$\frac{\partial U}{\partial \nu}\in
	L^2(\partial B_r)$ for a.e. $r\in(0,r_1)$.
\end{remark}
\begin{proof}
	For any $\phi \in C^\infty_0(0,r_0)$ we define $v(y):=\phi(|y|)$. Then we have 
\begin{align*}
  \int_0^{r_0}& H(r)\phi'(r)\, dy
         =\int_0^{r_0}\frac{1}{r^{N-1}}\bigg(
         \int_{\partial B_r}\mu U ^2 \,dS\bigg)\;\phi'(r)\, dr\\
       &=\int_{B^+_{r_0}}\frac{1}{|y|^{N}}\mu(y)U ^2 (y) \nabla v(y) \cdot y\, dy+
         \int_{B ^-_{r_0}}\frac{1}{|y|^{N}}\mu(y)U ^2 (y) \nabla v(y) \cdot y\, dy\\
       &=-\int_{B_{r_0}\setminus\tilde\Gamma}\frac{1}{|y|^{N}}(2\mu(y)v(y)U(y) \nabla U(y) \cdot y + v(y)U^2(y) \nabla \mu(y)\cdot y)\,dy\\
       &=-\int_0^{r_0}\frac{2}{r^{N-1}}\bigg(\int_{\partial B_r}\mu U
         \pd{U}{\nu} \,dS\bigg)\phi(r)\, dr-\int_0^{r_0}\frac{1}{r^{N-1}}\bigg(\int_{\partial B_r} U^2 \nabla \mu\cdot \nu \,dS\bigg)\phi(r)\, dr,
\end{align*}
	which proves \eqref{H'} thanks to \eqref{nabla-mu-O}.
	Since $r^{-N+1}$ is bounded in any compact subset of $(0,r_0]$,
        then, by \eqref{mu-estimates}, \eqref{nabla-mu-O} and the
        Coarea Formula, $H$ and $H'$ are locally integrable
so that $H \in W^{1,1}_{\rm loc}((0,r_0])$. 
\end{proof}

  Now we turn our attention to $E$. Henceforth we let $\{f_n\}$
be as in 
\eqref{fn1}, if $f$ satisfies \eqref{h1}, or as in \eqref{fn2}, if $f$
satisfies \eqref{h2}, and we consider the sequence $\{U_n\}$
converging to $U$ in $H^1(B_{r_0}\setminus\tilde \Gamma)$ provided
by  Theorem \ref{aprox theor}.

\begin{remark}\label{f-limits}
  By Proposition \ref{sobolev ineq} and \eqref{def eta}, $U_n\to U$ in
  $L^{q_\epsilon}(B_{r_0})$. Then, since $f_n\to \tilde f$ in
  $L^{\frac{N}{2}+\epsilon}(B_{r_0})$ by \eqref{eq:3}, from
  H\"older's inequality it easily follows that 
	\begin{equation}\label{limitL11}
		\lim_{n \to \infty}\int_{B_{r_0}}|\tilde f\,U^2- f_n\, U_n^2| \, dy=0.
	\end{equation}
Moreover, if $f$ satisfies \eqref{h1}, $\nabla \tilde f\in
   L^{\frac{N}{2}+\epsilon}(B_{r_0},\R^N)$ and hence
\begin{equation}\label{limitL12}
	\lim_{n \to \infty}\int_{B_{r_0}\setminus \Gamma}|(\nabla
        \tilde f \cdot \beta) \, (U^2-U_n^2)| \, dx=0,
\end{equation}
since the vector field $\beta$ defined in \eqref{mu-beta} is bounded
in view  \eqref{eq:27}.	
\end{remark}

\begin{lemma} \label{limitboundary}
  If $F_n \to F$ in $L^1(B_{r_0})$,
  then there exists a subsequence $\{F_{n_k}\}_{k \in \mathbb{N}}$
  such that, for a.e. $ r \in (0,r_0)$,
  \[
    \lim_{k \to \infty}\int_{\partial B_r}|F-F_{n_k}|\, dS=0 \quad
    \text{and} \quad \lim_{k \to \infty}\int_{S_{n_k,r}}F_{n_k}\,
    dS=\int_{\partial B_r} F\, dS,
  \]
 where the notation $S_{n,r}$ has been introduced in \eqref{Omega Gamma}.
\end{lemma}

\begin{proof}
	Let $h_n(r):=\int_{\partial B_r} |F_n-F| \, dS$. Since, by assumption and the Coarea Formula, 
	\[
	\lim_{n \to \infty}\int_{B_{r_0}}|F-F_n| \, dy=\lim_{n \to
		\infty}\int_0^{r_0}h_n(r)dr=0,
	\]
	we have that $h_n \to 0$ in $L^1(0,r_0)$. Hence there exists a
        subsequence $\{h_{n_k}\}_{k\in \mathbb{N}}$ converging to $0$
        a.e. in $(0,r_0)$. Therefore $F_{n_k} \to F$ in
        $L^1(\partial B_r)$ for a.e. $r \in (0,r_0)$.  It follows
        that, for a.e. $r \in (0,r_0)$,
	\[
	\int_{S_{n_k,r}}F_{n_k} \, dS -\int_{\partial B_r} F \,
	dS=\int_{\partial B_r} \chi_{S_{n_k,r}}(F_{n_k}-F)\, dS+
	\int_{\partial B_r} (\chi_{S_{n_k}}-1)F \, dS \to 0
	\]
	as  $k \to \infty$,
	thus yielding the conclusion.
\end{proof}

\begin{proposition}\label{p:der}
	We have that $E\in W_{\rm loc}^{1,1}((0,r_0])$, 
        
	\begin{equation}\label{E as H'}
		E(r)=\frac{1}{r^{N-2}}\int_{\partial B_r}U A \nabla U\cdot \nu \,dS=\frac{r}{2}H'(r) +rH(r)O(1)  \quad \text{as } r \to 0^+
		\end{equation}
	and 
	\begin{equation}\label{E'}
		E'(r)=(2-N)	\frac{1}{r^{N-1}}\int_{B_r \setminus \tilde \Gamma}(A \nabla U \cdot \nabla U-\tilde fU^2) \, dy+\frac{1}{r^{N-2}}\int_{\partial B_r}(A \nabla U \cdot \nabla U-\tilde fU^2) \, dS 
	\end{equation}
	in the sense of distributions and for a.e. $r\in (0,r_0)$.
\end{proposition}
\begin{proof}
  The fact that $E\in W^{1,1}_{\rm loc}((0,r_0])$ and \eqref{E'}
  follow from the Coarea Formula and \eqref{Found ineq}.  To prove
  \eqref{E as H'} we consider the sequence $\{U_n\}$
    introduced in Theorem \ref{aprox theor}.
For every $r\in(0,r_0)$ and
  $n \in {\mathbb N}\setminus\{0\}$,
	\[\frac{1}{r^{N-2}}\int_{\Omega_{n,r}}(A\nabla U_n\cdot \nabla
          U_n - f_n U_n^2) \, dy=\frac{1}{r^{N-2}}\int_{S_{n,r}}U_n A
          \nabla U_n\cdot \nu \,dS\] since $U_n$ solve \eqref{aprox
          prom} and $U_n \in H^2(\Omega_{n,r})$ by Theorem \ref{aprox
          theor}.  Thanks to Remark \ref{f-limits},the Dominated
        Convergence Theorem, and Lemma \ref{limitboundary}, we can
        pass to the limit, up to a subsequence, as $n \to \infty$ in
        the above identity for a.e. $r\in(0,r_0)$, thus proving the
        first equality in \eqref{E as H'}.  To prove the second
        equality in \eqref{E as H'} we define
	\[
          \zeta(y):=\frac{\mu(y)(\beta(y)-y)}{|y|}=\frac{A(y)y}{|y|}-\frac{A(y)y
            \cdot y}{|y|^3}y.
        \] 
       Then, since  $\zeta(y)\cdot y=0$ and
       $\zeta\cdot(0,\dots,0,1)=0$ on $\tilde\Gamma$, we have that
		\begin{align*}
                  \int_{\partial B_r}U A \nabla U\cdot \nu
                  \,dS-\int_{\partial B_r}\mu U\pd {U}{\nu} \, dS&=
                  \frac12\int_{\partial B_r}\zeta \cdot \nabla (U ^2)
                  \, dS\\
                  &=-\frac12 \int_{\partial B_r}\mathop{\rm{div}}(\zeta)U^2  \, dS
                  =r^{N-1}H(r)O(1)
                \end{align*}
as $r\to0$,	where we have used in the last equality the estimate
\[
  \mathop{\rm{div}}(\zeta)(y)=\left(\frac{\nabla
      \mu(y)}{|y|}-\frac{\mu(y)y}{|y|^3}\right)(\beta(y)-y)
  +\frac{\mu(y)}{|y|}\big(\mathop{\rm{div}}(\beta)(y)-N\big)=O(1)
\]
	which follows from Proposition \ref{beta-properties}. Then we conclude by \eqref{H'}.
	\end{proof}

The approximation procedure developed above also allows us to derive
the following integration by parts formula.
\begin{proposition}\label{p:intparts}
There exists a set $\mathcal M\subset [0,r_0]$ having null $1$-dimensional Lebesgue
  measure such that,  for all $r\in (0,r_0]\setminus\mathcal M$,  $A\nabla
    U\cdot\nu\in L^2(\partial B_r)$ and
  \begin{equation*}
    \int_{B_{r} \setminus \tilde\Gamma} A\nabla U\cdot\nabla \phi\,dx=
    \int_{B_{r}}\tilde f U\phi\, dx+\int_{\partial B_r} ( A\nabla
    U\cdot\nu)  \phi \, dS
  \end{equation*}
 for every $\phi\in H^1(B_{r_0}\setminus \tilde
  \Gamma)$, where $A\nabla
    U\cdot\nu$ on $\partial B_r$ is meant in the sense of Remark
    \ref{rem:gradsense}.
  \end{proposition}
  \begin{proof}
    Since $U_n\to U$ in $H^1(B_{r_0}\setminus\tilde\Gamma)$ in view of
    Theorem \ref{aprox theor}, by Lemma
    \ref{limitboundary} there exist a subsequence $\{U_{n_k}\}$ and
    a set $\mathcal M\subset [0,r_0]$
    having null $1$-dimensional Lebesgue measure such that
    $A\nabla
    U\cdot\nu\in L^2(\partial B_r)$ and 
    $A\nabla U_{n_k}\cdot\nu\to A\nabla U\cdot\nu$   in $L^2(\partial
    B_r)$ for all $r\in (0,r_0]\setminus\mathcal M$.
Since $U_n \in H^2(\Omega_{n,r})$ for any $r \in(0,r_0)$ and $n \in {\mathbb N}\setminus\{0\}$ by Theorem
\ref{aprox theor}, from \eqref{aprox prom} it follows that 
\[
  \int_{\Omega_{n,r}}(A\nabla U_n\cdot \nabla
          \phi - f_n U_n\phi) \, dy=\int_{S_{n,r}}\phi A
          \nabla U_n\cdot \nu \,dS.
        \]
        Arguing as in the proof of Proposition \ref{p:der}, we can pass
        to the limit along $n=n_k$ as $k\to\infty$ in the above identity for all
$r\in  (0,r_0]\setminus\mathcal M$, thus obtaining the conclusion.    
  \end{proof}

  \begin{theorem}(Pohozaev type inequality) 
	Under either  assumption \eqref{h1} or  assumption \eqref{h2}, for any
	$r \in (0,r_0]$ we have that 
	\begin{align}\label{Poho-ineq-1}
		&r\int_{\partial B_r} A\nabla U\cdot \nabla U\, dS \ge 
		2r\int_{\partial B_r}\frac{|A\nabla U\cdot \nu|^2}{\mu}\,dS+\int_{B_r\setminus \tilde \Gamma}(A\nabla U\cdot\nabla U) \mathop{\rm{div}}(\beta) \,dy \\
		&\quad+2\int_{B_r\setminus \tilde \Gamma}\frac{A\nabla U \cdot y  }{\mu} \tilde{f}\, U \,dy \notag
		+\int_{B_r\setminus \tilde \Gamma}(dA\nabla U	\nabla U)\cdot \beta \, dy
		-2\int_{B_r\setminus \tilde \Gamma} J_{\beta}(A\nabla U) \cdot \nabla U \, dy,
	\end{align}	
	which can be rewritten as 
	\begin{align}\label{Poho-ineq-2}
          &r\int_{\partial B_r} (A\nabla U\cdot \nabla U - \tilde f\,U^2)\, dS \ge 
            2r\int_{\partial B_r}\frac{|A\nabla U\cdot
            \nu|^2}{\mu}\,dS \\
          \notag 
          &\quad\quad +\int_{B_r\setminus \tilde \Gamma}(A\nabla
            U\cdot\nabla U)
            \mathop{\rm{div}}(\beta) \,dy +\int_{B_r\setminus \tilde
            \Gamma}
            (\tilde f \mathop{\rm{div}}(\beta)+\nabla \tilde f \cdot\beta )\,U^2\, dy\\ \notag
          &\quad\quad+\int_{B_r\setminus \tilde \Gamma}(dA\nabla U	\nabla U)\cdot \beta \, dy \notag 
            -2\int_{B_r\setminus \tilde \Gamma} J_{\beta}(A\nabla U) \cdot \nabla U \, dy
	\end{align}
	if $f$ satisfies  \eqref{h1}. 
\end{theorem}

\begin{proof}
  By Theorem \ref{aprox theor} we have that
  $U_n \in H^2(\Omega_{n,r})$ for any $r \in (0,r_0)$ and
  $n    \in {\mathbb N}\setminus\{0\}$. Then, since $A$ is symmetric by Proposition
  \ref{diffeomorphism}, we may write the following Rellich-Ne\u cas
  identity in a distributional sense in $\Omega_{n,r}$:
\begin{multline}\label{Rellich-Necas}
\mathop{\rm{div}}((A\nabla U_n\cdot \nabla U_n) \beta-2(\beta \cdot \nabla U_n) A\nabla U_n)
=(A\nabla {U_n}\cdot\nabla U_n) \mathop{\rm{div}}(\beta)\\
-2(\beta \cdot \nabla U_n)\mathop{\rm{div}}(A\nabla U_n)+(dA\nabla U_n	\nabla U_n)\cdot \beta -2J_{\beta}(A\nabla U_n) \cdot \nabla U_n.
\end{multline}
Since $U_n \in H^2(\Omega_{n,r})$ and the components of $A$ and
$\beta$ are Lipschitz {continuous by Propositions
\ref{diffeomorphism} and \ref{beta-properties}, then
$(A\nabla U_n\nabla U_n) \beta-2(\beta \cdot \nabla U_n) A\nabla U_n)
\in W^{1,1}(\Omega_{n,r})$. Therefore we can integrate
both sides of \eqref{Rellich-Necas}} on the
Lipschitz domain $\Omega_{n,r}$ and apply the Divergence
  Theorem to obtain,
in view of \eqref{mu-beta} and \eqref{aprox prom},
 \begin{align}\label{eq:1}
r\int_{S_{n,r}} &\left(A\nabla U_n\cdot \nabla U_n -2\frac{|A\nabla U_n\cdot \nu|^2}{\mu}\right)\,dS +\int_{ \Gamma_{n,r}} (A\nabla U_n\cdot \nabla U_n) \frac{Ay\cdot \nu }{\mu} \,dS\\
&=\int_{\Omega_{n,r}}(A\nabla {U_n}\cdot\nabla U_n) \mathop{\rm{div}}(\beta) \,dy +\notag
2\int_{\Omega_{n,r}}\frac{A\nabla U_n \cdot y  }{\mu} f_n U_n \,dy\\
&\quad\quad+\int_{\Omega_{n,r}}(dA\nabla U_n	\nabla U_n)\cdot \beta \, dy \notag
-2\int_{\Omega_{n,r}} J_{\beta}(A\nabla U_n) \cdot \nabla U_n \, dy.
\end{align}	
From Proposition \ref{negativity},
\eqref{ellipticity}, and \eqref{mu-estimates} it follows
  that, for all $n\in{\mathbb N}\setminus\{0\}$ and $r\in (0,r_0)$,
\begin{equation}\label{eq:29}
  \int_{ \Gamma_{n,r}} (A\nabla U_n\cdot \nabla U_n) \frac{Ay\cdot \nu
  }{\mu} \,dS \le 0.
\end{equation}
 We recall from  Theorem \ref{aprox theor} that $U_n \to
U$ strongly in $H^1(B_{r_0} \setminus \tilde
\Gamma)$, while Propositions \ref{diffeomorphism} and
\ref{beta-properties} imply that 
\begin{align} \label{everything-bounded} &\mu \in
  L^{\infty}(B_{r_0},\R), \quad \beta\in L^{\infty}(B_{r_0},\R^N),
                                           \quad \mathop{\rm{div}}{\beta}\in L^{\infty}(B_{r_0},\R),\\
                                         & A \in
                                           L^{\infty}(B_{r_0},\R^{N^2}),\quad
                                           \left\{\pd{a_{i,j}}{y_h}\right\}_{i,j,h=1,\dots,
                                           N}\in
                                           L^{\infty}(B_{r_0},\R^{N^3}).\notag
\end{align}
Furthermore, under assumption \eqref{h1}, we have that, by
  Sobolev embeddings (see Proposition \ref{sobolev ineq}), if $N\geq3$,
  then $f_n=\tilde f\in L^N(B_{r_0})$ and $U_n \to
U$ strongly in $L^{2^*}(B_{r_0})$, whereas, if $N=2$,
  then $f_n=\tilde f\in L^{2(1+\epsilon)/(1-\epsilon)}(B_{r_0})$ and $U_n \to
U$ strongly in $L^{(1+\epsilon)/\epsilon}(B_{r_0})$; then, since $\nabla U_n\to\nabla U$ in
$L^{2}(B_{r_0})$, H\"older's inequality ensures that
\begin{equation}\label{eq:28}
  f_n U_n A\nabla U_n\cdot y\to  \tilde f U A\nabla U\cdot y \quad\text{in } L^1(B_{r_0}).
\end{equation}
Under assumption \eqref{h2}, we have that Hardy's inequality (see
Proposition \ref{hardy}), Proposition \ref{trace r}  and \eqref{eq:5}  yield that
\[
  \int_{B_{r_0}}\left| f_n
    y(U_n-U)\right|^2\,dy\leq {\rm const\,}r_0^{4\epsilon} \int_{B_{r_0}}|y|^{-2}|U_n-U|^2\,dy
    \to 0\quad\text{as }n\to\infty
    \]
which, thanks to Proposition \ref{hardy} again and  the Dominated Convergence
Theorem, easily implies that 
\[
  f_n     yU_n\to \tilde fyU\quad\text{in }L^2(B_{r_0}),
\]
thus proving \eqref{eq:28} also under assumption \eqref{h2}.

Then, thanks to the Dominated Convergence
Theorem, \eqref{dA}, \eqref{eq:28} and Lemma
\ref{limitboundary}, we can pass to the limit in \eqref{eq:1} as $n \to\infty$, up to
a subsequence, and, taking into account \eqref{eq:29}, we obtain inequality
\eqref{Poho-ineq-1}.

If assumption \eqref{h1} holds then by \eqref{mu-beta},
  \eqref{fn1} and Proposition \ref{beta-properties}
  we have that
\begin{align}\label{eq:6}
  &
2\int_{\Omega_{n,r}}\frac{A\nabla U_n \cdot y}{\mu} f_n U_n \,dy=
  2\int_{\Omega_{n,r}}(\beta \cdot \nabla U_n)\tilde f U_n \,dy\\\quad&=
\notag  -
  \int_{\Omega_{n,r}}  (\tilde f \mathop{\rm{div}}(\beta)+\nabla \tilde f \cdot\beta )\,U_n^2\, dy+
 r\int_{S_{n,r}} \tilde f \,U_n^2\, dS +\int_{ \Gamma_{n,r}}\tilde f \, U_n^2\, \beta \cdot \nu \, dS.
\end{align}
We define
\begin{align*}
	&O^+_{n,r}:= O_n \cap B^+_r, \quad O^-_{n,r}:= O_n \cap B_r^-,\\
	& \Gamma^+_{n,r}:=\Gamma_{n,r}\cap B^+_r, \quad
	 \Gamma^-_{n,r}:= \Gamma_{n,r}\cap B^-_r,
\end{align*}
where $O_n$
  is defined in \eqref{eq:30}. Taking into account that
$\beta\cdot \nu =\frac{Ay }{\mu}\cdot \nu=0$ on
$\partial O^+_{n,r}\cap \partial \R^N_+$ since
$\nu=-(0,\dots,1)$ and \eqref{matrix-A} holds, the Divergence Theorem
yields that
\begin{multline}\label{eq:2}
  \int_{\Gamma^+_{n,r}}\tilde f U_n^2 \,\beta \cdot \nu \, dS
  =-r\int_{\partial O^+_{n,r}\cap \partial B_r}\tilde f U_n^2 \,\beta
  \cdot \nu \, dS\\+ \int_{O^+_{n,r}}\left(\tilde f
    U_n^2\mathop{\rm{div}{\beta}}+ U_n^2\,\nabla \tilde f \cdot \beta
    \,+2\tilde f U_n\nabla U_n \cdot \beta \right)\, dy.
\end{multline}
By \eqref{limitL11}, \eqref{everything-bounded}, and  Lemma
\ref{limitboundary} there exists a subsequence $\{\tilde f
\,U^2_{n_k}\beta \cdot \nu \}_{k \in \mathbb{N}}$ converging in
$L^1(\partial B_r)$ and hence equi-integrable in $\partial B_r$ for a.e. $r \in (0,r_0)$, hence
\[
  \lim_{k \to \infty}\int_{\partial O^+_{n_k,r}\cap \partial B_r}\tilde f\,U_{n_k}^2\beta \cdot\nu\, dS=0 \quad \text{for a.e. } r \in (0,r_0).\]  
Since $\nabla U_n \to \nabla U$ in $L^2(B_{r_0}^+, \R^N)$, $U_n \to U $ in
$L^{q_{\epsilon}}(B_{r_0}^+)$ and $\tilde f\in L^{N+2 \epsilon}(B_{r_0}^+)$ by \eqref{h1} and classical Sobolev
embeddings, from  \eqref{everything-bounded}  and H\"older's inequality we deduce that 
\[
  \tilde f U_n\nabla U_n \cdot \beta  \to  \tilde f U\nabla U \cdot \beta \quad\text{in }
L^1(B_{r_0}^+),
\]
so that  $\{\tilde f U_n\nabla U_n \cdot \beta\}_{n \in \mathbb{N}}$
is equi-integrable in $B_{r_0}^+$.  Therefore 
\[
  \lim_{n \to \infty}\int_{O_{n,r}^+}\tilde f U_n\nabla U_n \cdot
  \beta \, dy=0\quad\text{for all }r\in(0,r_0).
\]
Moreover, also
$\{\mathop{\rm{div}{\beta}}\tilde f \,U_n^2+ U_n^2\,\nabla \tilde f
\cdot\beta \}_{n \in \mathbb{N}}$ is equi-integrable thanks to
\eqref{limitL11} and \eqref{limitL12}. It follows that
\[
  \lim_{n \to \infty}\int_{O_{n,r}^+}(\mathop{\rm{div}{\beta}}\tilde
  f U_n^2+ \nabla \tilde f\cdot\beta\,U_n^2)\, dy=0\quad\text{for
  all }r\in(0,r_0).
\]
Then from \eqref{eq:2} we  conclude that 
\[
  \lim_{k \to \infty} \int_{\Gamma_{n_k,r}^+}\tilde f U_{n_k}^2 \beta
  \cdot \nu \, dS=0.
\]
In a similar way we obtain that $\lim_{k \to \infty} \int_{\Gamma_{n_k,r}^-}\tilde f U_{n_k}^2 \beta
  \cdot \nu \, dS=0$ so that
\[
\lim_{k \to \infty} \int_{\Gamma_{n_k,r}}\tilde f U_{n_k}^2 \beta
\cdot \nu \, dS=0.
\]
Therefore \eqref{Poho-ineq-2} follows by passing to the limit in
\eqref{eq:1} and \eqref{eq:6} as $n\to\infty$ along a subsequence,
taking into account Proposition \ref{negativity}, the Dominated
Convergence Theorem, \eqref{dA}, Remark \ref{f-limits} and Lemma
\ref{limitboundary}.
\end{proof}

\begin{proposition} \label{E ineq}
	For a.e. $r\in(0,r_0)$
	\begin{align}\label{E ineq 1}
          E'(r) &\ge 
                  2r^{2-N}\int_{\partial B_r}\frac{|A\nabla U\cdot \nu|^2}{\mu}\,dS +r^{1-N}\int_{B_r\setminus \tilde \Gamma} (\mathop{\rm{div}}(\beta)+2-N) A\nabla U\cdot\nabla U \,dy\\ \notag 
		&\quad \quad+r^{1-N}\int_{B_r\setminus \tilde \Gamma}  \big(\tilde f (\mathop{\rm{div}}(\beta)+N-2)+\nabla \tilde f \cdot\beta\big)\,U^2\, dy\\ \notag
		&\quad \quad +r^{1-N}\int_{B_r\setminus \tilde \Gamma}(dA\nabla U	\nabla U)\cdot \beta \, dy \notag 
           -2r^{1-N}\int_{B_r\setminus \tilde \Gamma} J_{\beta}(A\nabla U) \cdot \nabla U \, dy,
	\end{align}
	if \eqref{h1} holds, and 
\begin{multline}\label{E ineq 2}
	E'(r) \ge 
	2r^{2-N}\int_{\partial B_r}\frac{|A\nabla U\cdot \nu|^2}{\mu}\,dS -r^{2-N}\int_{\partial B_r}\tilde f U^2 \, dS 	+(N-2)r^{1-N}\int_{B_r}\tilde f U ^2 \, dy\\ 
    +r^{1-N}\int_{B_r\setminus \tilde \Gamma}(A\nabla \bar  u\cdot\nabla U) (\mathop{\rm{div}}(\beta)+2-N) \,dy\
	 +2r^{1-N}\int_{B_r\setminus \tilde \Gamma}\frac{A\nabla U \cdot y  }{\mu} \tilde{f} U \,dy \\ 
+r^{1-N}\int_{B_r\setminus \tilde \Gamma}(dA\nabla U	\nabla U)\cdot \beta \, dy 
	-2r^{1-N}\int_{B_r\setminus \tilde \Gamma} J_{\beta}(A\nabla U) \cdot \nabla U \, dy\end{multline}
	if \eqref{h2} holds.
\end{proposition}

\begin{proof}
	Estimates \eqref{E ineq 1}--\eqref{E ineq 2} are direct
        consequences of  \eqref{E'}, \eqref{Poho-ineq-1}, and \eqref{Poho-ineq-2}.
\end{proof}

We now introduce the Almgren frequency function, defined as  
\begin{equation}\label{eq:43}
\mathcal{N}:(0,r_0] \rightarrow \mathbb{R}, \quad
\mathcal{N}(r):=\frac{E(r)}{H(r)}.
\end{equation}
The above definition of $\mathcal{N}$ is well posed thanks to  Proposition \ref{H positive}. 

\begin{proposition} \label{N'} If either assumption \eqref{h1} or
	assumption \eqref{h2} hold, then
	$\mathcal{N} \in W^{1,1}_{\rm loc}((0,r_0])$ and, for any $r \in (0,r_0]$,
	\begin{equation}\label{N lower bound}
		\mathcal{N}(r)\ge -2\eta_{\tilde f}(r). 
	\end{equation}
	Furthermore, for a.e. $r \in (0,r_0)$,
        \begin{equation}\label{N' lower bound}
		\mathcal{N}'(r)\ge \mathcal V(r) + \mathcal W(r)
	\end{equation}
	where 
	\begin{equation}\label{v}
          \mathcal V(r)=
          \frac{2r\left(\Big(\int_{\partial B_r}\frac{|A\nabla U\cdot
                \nu|^2}{\mu}\, dS\Big)
              \Big(\int_{\partial B_r}\mu U^2dS\Big)-
              \left(\int_{\partial B_r}U A \nabla U\cdot \nu \,dS\right)^2\right)}{\left(\int_{\partial B_r}\mu U^2\, dS\right)^2}\geq0
	\end{equation}
	and 
	\begin{equation}\label{w1}
		 \mathcal W(r)=O \left(r^{-1+\frac{4\epsilon}{N+2\epsilon}}\right)(1+\mathcal{N}(r))
                 \quad\text{as $r\to0^+$.}
               \end{equation}
\end{proposition}

\begin{proof}
	Since $1/H,E \in W_{\rm loc}^{1,1}((0,r_0])$, then $\mathcal{N} \in W_{\rm loc}^{1,1}((0,r_0])$.
Furthermore \eqref{nabla ineq} directly implies~\eqref{N
    lower bound}.

By\eqref{E as H'}, for a.e.  $r \in (0,r_0)$
\begin{align}
\label{eq:31}  \mathcal{N'}(r)&=\frac{E'(r)H(r)-E(r)H'(r)}{H^2(r)}=\frac{E'(r)H(r)-\frac{2}{r}E^2(r)}{H^2(r)}+\frac{E(r)O(1)}{H(r)}\\
\notag&=\frac{E'(r)H(r)-\frac2r r^{4-2N}\big(\int_{\partial B_r}U A
        \nabla U\cdot \nu \,dS\big)^2}{H^2(r)}+O(1)\mathcal N(r)
\end{align}  
as $r\to0^+$.
	By Proposition \ref{diffeomorphism}, Proposition \ref{beta-properties}, \eqref{def eta} and \eqref{nabla ineq} 
\begin{align*}
	&\left|\int_{B_r\setminus \tilde \Gamma}\Big((A\nabla
   U\cdot\nabla U) (\mathop{\rm{div}}(\beta)+2-N)-2J_{\beta}(A\nabla
   U) \cdot \nabla U +(dA\nabla U	\nabla U)\cdot \beta\Big)\,
   dy\right|\\
                                                         &\quad\le
	 O(r)\int_{B_r\setminus \tilde \Gamma} |\nabla U| ^2 \, dy \\
	 &\quad\le O(r)\int_{B_r\setminus\tilde\Gamma}(A \nabla U \cdot
           \nabla U-\tilde fU^2) \, dy +O\left(r^{\frac{4
           \epsilon}{N+2\epsilon}}\right) \int_{\partial B_r}\mu U^2
           \,dS\quad\text{as }r\to0^+.
\end{align*} 
By  \eqref{Found ineq},  \eqref{nabla ineq},  and \eqref{mu-estimates}
\begin{align*}
\int_{B_r}\tilde f U ^2 \, dy &\le O\left( r^{\frac{4 \epsilon}{N+2	\epsilon}}\right)\int_{B_r\setminus\tilde\Gamma} |\nabla U |^2\, dy +O\left(r^{\frac{2\epsilon -N}{N+2\epsilon}}\right) \int_{\partial B_r} U^2 \,dS\\
&\le  O\left(r^{\frac{4\epsilon}{N+2\epsilon}}\right)\int_{B_r\setminus\tilde\Gamma}(A \nabla U \cdot \nabla U \, dy-\tilde {f}U^2)+O\left(r^{\frac{2\epsilon -N}{N+2\epsilon}}\right) \int_{\partial B_r}\mu U^2 \,dS
\end{align*}
as $r\to0^+$ and, by \eqref{div-beta}, the same holds for   $\int_{B_r}(\mathop{\rm{div}}{\beta}-N+2)\tilde f U ^2 \, dy $.
In the same way from \eqref{eq:27}   it follows that, if \eqref{h1} holds,
\begin{equation*}
	\int_{B_r}\nabla \tilde f \cdot \beta  U ^2 \, dy \le {O\left(r^{\frac{4\epsilon}{N+2\epsilon}}\right)}\int_{B_r\setminus\tilde\Gamma}(A \nabla U \cdot \nabla U \, dy-\tilde {f}U^2)+O\left(r^{\frac{2\epsilon -N}{N+2\epsilon}}\right) \int_{\partial B_r}\mu U^2 \,dS
      \end{equation*}
      as $r\to0^+$.

      Under assumption \eqref{h2}, by Remark \ref{tilde-f-regularity},
      \eqref{mu-O}, \eqref{A-oper-norm}, \eqref{def eta} \eqref{Found ineq}, \eqref{nabla ineq} and H\"older's inequality,
\begin{align*}
  &\int_{B_r\setminus \tilde \Gamma}\frac{A\nabla U \cdot y  }
    {\mu} \tilde{f} U \,dy=O(r)\int_{B_r\setminus \tilde \Gamma}|\nabla U|  |\tilde{f}| U \,dy \notag\\
  &\le O(r^\e)\norm{\nabla U}_{L^2(B_r\setminus \tilde  \Gamma)}
    \left(\int_{B_r} |\tilde f| U^2\, dx\right)^{\frac{1}{2}}\\
  &\le O\left(r^{\e+\frac{2\e}{N+2 \e}}\right)
    \left(\int_{B_r \setminus  \tilde \Gamma}(A \nabla  U\cdot \nabla U -\tilde fU^2)\, dy+ 
    \frac{2}{\eta_f(r)}{r} \int_{\partial B_r}\mu U^2 \, dS\right)^{\frac{1}{2}}\times\\ \notag 
  \notag &\qquad\times\left( \int_{B_r \setminus\tilde  \Gamma}
           (A \nabla  U\cdot \nabla U -\tilde fU^2)\, dy+
           \frac{2}r \int_{\partial B_r}\mu U^2 \, dS \right)^{\frac{1}{2}}\\
  \notag  &\le O\left(r^{\e+\frac{2\e}{N+2 \e}}\right)
            \int_{B_r \setminus \tilde \Gamma}(A \nabla  U\cdot \nabla
            U -\tilde fU^2)\, dy
            +O\left(r^{-1+\e+\frac{2\e}{N+2 \e}}\right)
            \int_{\partial B_r}\mu U^2 \, dS. 
\end{align*}
Under assumptions \eqref{h2}, thanks to Remark
\ref{tilde-f-regularity} and \eqref{mu-estimates},
\begin{equation*}
\int_{\partial B_r}\tilde f U^2 \, dS=O\left(r^{2\e-2}\right)\int_{\partial B_r}\mu U^2 \, dS.
\end{equation*}
Collecting the above estimates, we conclude that
\eqref{N' lower bound}, \eqref{v} and \eqref{w1} follow from \eqref{E
  ineq 1}or \eqref{E ineq 2} under hypotheses \eqref{h1} or \eqref{h2}
respectively.  From the Cauchy–Schwarz inequality we also deduce that
$\mathcal V\geq 0$ a.e. in $(0,r_0)$.
\end{proof}	

We now prove that $\mathcal{N}$ is bounded.
\begin{proposition}
	There exists  a constant $C > 0$ such that, for  every $r \in (0,r_0]$,  
	\begin{equation}\label{N upper bound}
		\mathcal{N}(r) \le C.
	\end{equation}
\end{proposition}

\begin{proof}
	By Proposition \ref{N'} there exists a constant $\kappa>0$  such that, for a.e. $r \in  (0,r_0)$, 
	\[
          (\mathcal{N}+1)'(r) \ge \mathcal W(r)\ge -\kappa\,
          r^{-1+\frac{4\epsilon}{N+2\epsilon}}(\mathcal{N}(r)+1).
        \]
        Since
        $\mathcal{N}+1 > 0$ by \eqref{N lower bound} and the choice of
        $r_0$ in \eqref{beta1}, it follows that
	\[
	(\log(\mathcal{N}+1))' \ge -\kappa r^{-1+\frac{4\epsilon}{N+2\epsilon}}.
	\]
	An integration over $(r,r_0)$ yields
	\[\mathcal{N}(r) \le -1+\exp\left(\kappa\frac{N +2
		\epsilon}{4\epsilon}r_0^{\frac{4 \epsilon}{2 \epsilon
			+N}}\right)(\mathcal{N}(r_0)+1)\]
	and the proof is thereby complete.
\end{proof}
  \begin{proposition}\label{existence limit gamma}
	There exists the limit 
	\begin{equation}\label{limit gamma }
		\gamma:=\lim_{r \to 0^+}\mathcal{N}(r).
	\end{equation}
	Furthermore $\gamma$ is finite and $\gamma \ge 0$.
\end{proposition}

\begin{proof}
  From Proposition \ref{N'} and \eqref{N upper bound} there exists a
  constant $\kappa>0$ such that
\[
\mathcal{N}'(r)\ge \mathcal W(r) \ge -\kappa \,r^{-1+\frac{4
    \epsilon}{N+2\epsilon}}(\mathcal{N}(r)+1) \ge -\kappa(C+1)
r^{-1+\frac{4
    \epsilon}{N+2\epsilon}} \quad \text{for a.e. } r \in
(0,r_0).
\] 
Then
	\[
	\frac{d}{dr}\left(
	\mathcal N(r)+\frac{\kappa(C+1)(N+2\epsilon)}{4\epsilon}r^{\frac{4\epsilon}{N+2\epsilon}}
	\right)\ge 0
	\]
	for a.e. $r\in(0,r_0)$.
	We conclude that $\lim_{r \to 0^+}\mathcal{N}(r)$ exists; moreover
	such a limit is finite thanks to \eqref{N upper bound} and \eqref{N
		lower bound}. 
	Furthermore from \eqref{def eta} and \eqref{N lower bound} we deduce  that $\gamma \ge 0$.
\end{proof}

\begin{proposition}
	There exists a constant $\alpha > 0$ such that, for  every $r \in (0,r_0]$, 
	\begin{equation}\label{H upper bound}
		H(r) \le \alpha \,r^{2 \gamma}.
	\end{equation}
	Furthermore for  every $\sigma > 0$ there exist $\alpha_\sigma > 0$ and $r_\sigma \in (0,r_0)$ such that, for  every $r\in (0,r_\sigma]$,
	\begin{equation}\label{H lower bound}
		H(r) \ge \alpha_\sigma r^{2 \gamma+\sigma}.
	\end{equation}
\end{proposition}

\begin{proof}
	By \eqref{E as H'}, \eqref{N' lower bound}, \eqref{v},
        \eqref{w1}, and \eqref{N upper bound}  there exists a constant $\kappa>0$  such
        that, for a.e. $r \in  (0,r_0)$,
	
\begin{align*}
  \frac{H'(r)}{H(r)}&=\frac{2}{r}\mathcal{N}(r){+O(1)}\\
                    &=\frac{2}{r}
                      \int_0^r\mathcal{N}'(t)\, dt+\frac{2\gamma}{r}
                      +O(1)\ge \frac{2}{r} \int_0^r\mathcal W(t) \,dt
                      +O(1)+\frac{2\gamma}{r}\ge
                      -\kappa\, r^{-1+\frac{4\epsilon}{N+2\epsilon}}+\frac{2\gamma}{r}.
\end{align*}
	It follows that, integrating between $r$ and $r_0$, 
	\[H(r_0)\ge H(r)\left(\frac{r_0}{r}\right)^{2\gamma} 
	\exp\left(-\kappa\tfrac{N+2\epsilon}{4\epsilon}r_0^{\frac{4
			\epsilon
		}{N+2\epsilon}}\right)
	\]
	for all $r\in(0,r_0]$, so that \eqref{H upper bound} is
        proved.  To prove\eqref{H lower bound} we notice that, by
        \eqref{limit gamma } and \eqref{E as H'}, for every $\sigma > 0$ there exists
        $r_\sigma \in (0,r_0)$ s.t.
\[
\frac{H'(r)}{H(r)} \le \frac{2\gamma + \sigma}{r}\quad\text{for all }r\in(0,r_\sigma]
\]
	and an integration  between $r \in (0,r_\sigma)$ and $r_\sigma$ yields
	\[H(r_\sigma)\le\Bigl(\frac{r_\sigma}{r}\Bigr)^{2\gamma+\sigma} H(r)
	\]
	thus proving \eqref{H lower bound}.
\end{proof}

\begin{proposition}
	The limit 
	$\lim_{r \to 0^+}r^{-2 \gamma} H(r)$
	exists and is finite.
\end{proposition}
\begin{proof}
  By \eqref{H upper bound} we only need to prove that the limit
  exists.  For any $r \in (0,r_0)$  we have that
	\[
\frac{d}{dr}\frac{H(r)}{r^{2\gamma}}= \frac{2r^{2 \gamma
    -1}E(r)-2\gamma r^{2\gamma-1}H(r)+r^{2\gamma}H(r)O(1)}{r^{4\gamma}}=2r^{-2 \gamma
  -1}H(r)\left(\int_0^r\mathcal{N}'(t) \, dt+rO(1)\right),
\]
	by \eqref{E as H'} and Proposition \ref{existence limit gamma}. 
	Thanks to Proposition \ref{N'},  integrating between $r$ and $r_0$ we obtain 
	\begin{align}\label{H limit}
          \frac{H(r_0)}{r_0^{2\gamma}}-\frac{H(r)}{r^{2\gamma}}=
&\int_r^{r_0}2s^{-2 \gamma -1}H(s)\left(\int_0^s
  (\mathcal{N}'(t)-\mathcal W(t))\, dt\right) ds\\
          \notag &+\int_r^{r_0}2s^{-2 \gamma -1}H(s)\left(sO(1)+\int_0^s
                   \mathcal W(t)\, dt\right)ds.
	\end{align}
	We note that  there exists a constant $\kappa>0$  such that 
	\[
\left|2s^{-2 \gamma -1}H(s)\left(sO(1)+\int_0^s \mathcal
    W(t)\, dt\right)\right| \le
\kappa\,s^{-1+\frac{4\epsilon}{N+2 \epsilon}} \]
	by Proposition \ref{N'}, \eqref{N upper bound}, and \eqref{H
          upper bound}.
 Since $s^{-1+\frac{4\epsilon}{N+2 \epsilon}}\in
  L^1(0,r_0)$, then 
\begin{equation*}
\lim_{r\to0^+}\int_r^{r_0}2s^{-2 \gamma -1}H(s)\left(sO(1)+\int_0^s
                   \mathcal W(t)\, dt\right)ds
    \end{equation*}
exists and is finite.
	Moreover, since $\mathcal{N}'-\mathcal W\ge 0$  by Proposition \ref{N'},
	\[\lim_{r\to 0^+}\int_r^{r_0}2s^{-2 \gamma -1}H(s)\left(\int_0^s
	(\mathcal{N}'(t)-\mathcal W(t))\, dt\right) ds\] exists, being possibly
	infinite. Then the right hand side of \eqref{H limit} admits a limit
	as $r\to0^+$ and the conclusion follows.
\end{proof}

  From the properties of the height function $H$ derived above, in particular from estimate \eqref{H lower bound}, we
  deduce the unique continuation property stated in Theorem  \ref{t:ucp}.
  \begin{proof}[Proof of Theorem \ref{t:ucp}]
     	Let $u$ be a weak solution to \eqref{op} such that
        $u(x)=O(|x|^k)$ as $|x| \to  0^+$ for all $k
        \in\mathbb{N}$. To prove that 	$u\equiv0$ in $B_R$, we argue
        by contradiction and assume that $u\not\equiv 0$. Then we can
        define a frequency function for $U=u \circ F$ as in \eqref{H}, \eqref{E} and
        \eqref{eq:43}. Choosing $k\in{\mathbb N}$ such that
        $k>\gamma+\frac\sigma2$, we would obtain that $H(r)=O(r^{2k})=o(r^{2\gamma+\sigma})$
        as $r\to 0$, contradicting estimate~\eqref{H lower bound}.
  \end{proof}

\section{Neumann eigenvalues on $\mathbb{S}^{N-1} \setminus \Sigma$}
\label{sec:neum-eigenv-mathbbsn}
In this section we study the spectrum of \eqref{sphereprob}.
 We recall that $\mu\in\R$ is an eigenvalue of
  \eqref{sphereprob} if there exists
  $\psi\in H^1(\mathbb S^{N-1}\setminus\Sigma)\setminus\{0\}$ such
  that
\begin{equation}\label{weak-eig}
\int_{\mathbb S^{N-1}\setminus\Sigma}\nabla_{\mathbb S^{N-1}\setminus\Sigma}\psi\cdot\nabla_{\mathbb S^{N-1}\setminus\Sigma}\phi\,dS=
\mu\int_{\mathbb S^{N-1}\setminus\Sigma}\psi\phi\,dS\quad\text{for
  any }\phi\in  H^1(\mathbb S^{N-1}\setminus\Sigma).
\end{equation}
A Rellich-Kondrakov type theorem is needed to apply the classical
Spectral Theorem to problem~\eqref{sphereprob}.
\begin{proposition} \label{Rellich}
	The embedding  $H^1(\mathbb{S}^{N-1} \setminus \Sigma) \hookrightarrow L^2(\mathbb{S}^{N-1})$ is compact. 
\end{proposition}
\begin{proof}
  Let $\{\phi_n\}_{n \in
    \mathbb{N}}$ be a bounded sequence in
  $H^1(\mathbb{S}^{N-1} \setminus
  \Sigma)$.  We observe that
  $\mathbb{S}^{N-1}_+$ and
  $\mathbb{S}^{N-1}_-$ are smooth compact manifolds with boundary and
  that the sequences of restrictions
    $\big\{\phi_n\big|_{\mathbb{S}^{N-1}_+}\big\}_{n \in
      \mathbb{N}}$ and
    $\big\{\phi_n\big|_{\mathbb{S}^{N-1}_-}\big\}_{n \in
      \mathbb{N}}$ are bounded in
  $H^1(\mathbb{S}_+^{N-1})$ and
  $H^1(\mathbb{S}_-^{N-1})$ respectively. Then we can
  extract a subsequence $\{\phi_{n_k}\}_{k \in
    \mathbb{N}}$ such that $\big\{\phi_n\big|_{\mathbb{S}^{N-1}_+}\big\}_{n \in
      \mathbb{N}}$ converges in $L^2(\mathbb
  S_+^{N-1})$ by the classical Rellich-Kondrakov Theorem on compact
  manifolds with boundary, see \cite{aubin}. Proceeding
  in the same way for $\big\{\phi_{n_k}\big|_{\mathbb{S}^{N-1}_-}\big\}_{n \in
      \mathbb{N}}$ in
  $H^1(\mathbb{S}_-^{N-1})$, we conclude that there exists a
  subsequence $\{\phi_{n_{{k}_h}}\}_{h \in
    \mathbb{N}}$ which converges both in
  $L^2(\mathbb{S}_-^{N-1})$ and in
  $L^2(\mathbb{S}_+^{N-1})$, hence in $L^2(\mathbb{S}^{N-1})$.
\end{proof}

\begin{proposition} \label{eigen}\quad

  \begin{enumerate} 
  \item[(i)]The point spectrum of \eqref{sphereprob} is
    a diverging and increasing sequence of non-negative eigenvalues
    $\{\mu_k\}_{k \in \mathbb{N}}$ of finite multiplicity and the
    eigenvalue $\mu_0=0$ is simple. Letting $N_k$ be the multiplicity of
    $\mu_k$ and $V_k$ be the eigenspace associated to $\mu_k$, there
    exists an orthonormal basis of $L^2( \mathbb{S}^{N-1})$ consisting
    of eigenfunctions $\{Y_{k,i}\}_{k \in \mathbb{N},i=1,\dots, N_k}$
    such that $\{Y_{k,i}\}_{i=1,\dots N_k}$ is a basis of $V_k$ for
    any $k \in \mathbb{N}$.
		\item[(ii)] For any  $k \in \mathbb{N}$
		\begin{equation}\label{muk}
			\mu_k=\frac{k(k+2N-4)}{4}. 
		\end{equation}  
		Moreover any eigenfunction of \eqref{sphereprob}
                belongs to $L^\infty(\mathbb{S}^{N-1})$.
	\end{enumerate} 
\end{proposition}

\begin{proof}
  The proof of (i) follows from the classical Spectral Theorem for
  compact self-adjoint operators, taking into account Proposition
  \ref{Rellich}.  We prove now (ii). If $\mu$ is an eigenvalue of
  \eqref{sphereprob} and $\Psi$ an associated eigenfunction, let
  $\sigma:=-\frac{N-2}{2}+\sqrt{\left(\frac{N-2}{2}\right)^2+\mu}$ and
  \[W(r \theta):=r^\sigma\Psi(\theta), \quad \text{for any } r \in [0,
    \infty), \,\theta \in \mathbb{S}^{N-1}\setminus \Sigma. \] Since
  $\Psi$ is an eigenfunction of \eqref{sphereprob} then $W$ is
  harmonic on $B_1 \setminus \tilde \Gamma$ and
  $\pd {^+W}{\nu^+}=\pd {^-W}{\nu^-}=0$ on $\tilde \Gamma$.  Therefore
  we deduce from \cite{costabel2003asymptotics} that there exists
  $k \in \mathbb{N}$ such that $\sigma = \frac{k}{2}$ and so
  $\mu=\frac{k(k+2N-4)}{4}$. Moreover from
  \cite{costabel2003asymptotics} it also follows that
  $W \in L^\infty(B_1)$ hence $\Psi \in L^{\infty}(\mathbb{S}^{N-1})$.
	
Viceversa,  if we let $k\in\mathbb N$ and define $W$  in cylindrical coordinates  as
\[
  W(x',r\cos(t), r\sin(t)):=r^\frac{k}{2} \cos\left(\frac{k
      t}{2}\right) \quad \text{for any } x' \in \R^{N-2}, \ r \in
  [0,\infty), \text{ and } t \in [0,2 \pi],
\] 
then $W$ is harmonic on
$B_1 \setminus  \tilde \Gamma$  and $\pd {^+W}{\nu^+}=\pd
{^-W}{\nu^-}=0$
on $\tilde \Gamma$. Since $W$ is homogeneous of degree $k /2$, then 
\[
  W(r \theta)=r^{\frac{k}{2}} \Psi(\theta), \quad \text{for any }  r
  \in [0, \infty), \text{ and }\theta \in \mathbb{S}^{N-1} \setminus
  \Sigma,
\] 
where $\Psi=W_{|_{\mathbb S^{N-1}}}$. Then from
\[
  r^{\frac{k-4}{2}}\left(\frac{k(k-2)}{4}\Psi(\theta)
    +\frac{k(N-1)}{2} \Psi(\theta) +\Delta_{\mathbb S^{N-1}}
    \Psi(\theta)\right)=0,\quad r \in [0, \infty), \ \theta \in
  \mathbb S^{N-1} \setminus \Sigma,
\]
we deduce that $\Psi$ solves \eqref{sphereprob} with $\mu =\frac{k(k+2N-4)}{4}$.
\end{proof}

\begin{remark}\label{rem:eigenfunctionnonzero}
  The traces of eigenfunctions of problem \eqref{sphereprob} on both
  sides of 
  $\Sigma$ (i.e. the traces of 
  restrictions to $\mathbb{S}^{N-1}_+$ and $\mathbb{S}^{N-1}_+$)
  cannot vanish identically.

  Indeed, if an eigenfunction $\Psi$
  associated to the eigenvalue $\mu_k$ is such that the trace of
  $\Psi\big|_{\mathbb{S}^{N-1}_+}$ on $\Sigma$ vanishes, then the
  function $W(x):=|x|^{k/2}\Psi(x/|x|)$
  would be a harmonic function in $\R^{N}\setminus\tilde\Gamma$
  satisfying both Dirichlet and Neumann homogeneous boundary
  conditions on the upper side of the crack, thus violating classic
  unique continuation principles.
\end{remark}

\section{The blow-up analysis} 
\label{sec:blow-up-analysis}
Throughout this section we let $u\in H^1(B_R\setminus\Gamma)$
be a non-trivial weak solution to \eqref{op} with $f$ satisfying either \eqref{h1} or \eqref{h2}, $U=u \circ F\in
H^1(B_{r_1}\setminus\tilde\Gamma)$ be the corresponding solution to
\eqref{straightenedproblem}, $r_0$ be as in
\eqref{beta1} and $r_1$ be as in Proposition \ref{diffeomorphism}.  For all $\lambda\in (0,r_0)$, let 
\begin{equation}\label{wla}
  W^{\la}(y):=\frac{U(\la y)}{\sqrt{H(\la)}}  \quad \text{for any }
  y \in B_{\lambda^{-1}r_1} \setminus \tilde \Gamma.	
\end{equation}
For any $\la \in (0,r_0)$ it is easy to verify that
$W^{\la} \in H^1(B_{\lambda^{-1}r_1}\setminus \tilde \Gamma)$ and $ W^{\la}$  satisfies

\begin{equation}\label{wla equa}
  \int_{B_{\lambda^{-1}r_1} \setminus \tilde  \Gamma}
 A(\lambda y)\nabla  W^\la(y)
  \cdot \nabla \phi(y) \, dy -\lambda^2\int_{B_{\lambda^{-1}r_1}}
\tilde f(\lambda y) W^\la(y) \phi(y) \, dy=0
\end{equation}
for any $\phi \in H^1_{0,\partial B_{\lambda^{-1}r_1}}(B_{\lambda^{-1}r_1} \setminus \tilde \Gamma)$.  
In other words $ W^\la$ is a weak solution of 
\begin{equation}\label{wla prob}
  \begin{cases}
		-\mathop{\rm{div}}(A(\lambda\cdot)\nabla
                W^\lambda)=\lambda^2
                \tilde f(\lambda \cdot) W^\lambda, &\text{in } B_{\lambda^{-1}r_1} \setminus \tilde \Gamma,  \\[5pt]
		A(\lambda\cdot)\nabla^+  W^\lambda \cdot
                \nu^+=A(\lambda\cdot)
                \nabla^-  W^\lambda\cdot \nu^-=0,
		&\text{on } \tilde\Gamma,			
	\end{cases}
\end{equation}
for any $\la \in (0,r_0)$. Since $B_{1}\subset B_{\lambda^{-1}r_1}$ for all $\la \in (0,r_0)$, it 
follows that, for any $\la \in (0,r_0)$,
\begin{equation}\label{wla equa1}
\int_{B_1 \setminus \tilde \Gamma} A(\lambda y)\nabla  W^\la(y) \cdot \nabla \phi(y) \, dy -\lambda^2\int_{B_1}\tilde f(\lambda y )  W^\la(y) \phi(y) \, dy=0,
\end{equation}
for any $\phi \in H^1_{0,\partial B_1}(B_{1} \setminus \tilde \Gamma)$. 
Furthermore by a change of variables, \eqref{wla} and \eqref{H},
\begin{equation}\label{wla bounded boundary}
\int_{\mathbb{S}^{N-1}}\mu(\lambda\theta) | W^\la(\theta)|^2 dS=1 \quad \text{for every } \la \in (0,r_0).
\end{equation}

\begin{proposition} \label{wla bounded}
Let $ W^\la$ be  as in \eqref{wla}. Then $\{ W^\la\}_{\la \in (0,r_0)}$ is bounded in $H^1( B_1 \setminus \tilde{\Gamma})$.
\end{proposition}
\begin{proof}
We have 
\begin{equation*}
  \int_{ B_1 \setminus \tilde{\Gamma}} |\nabla  W^\la|^2 \, dy= \frac{\la^{2-N}}{H(\la)}\int_{ B_\lambda \setminus \tilde \Gamma} |\nabla U(y)|^2 \, dy \le 
  \frac{2}{1-2\eta_{\tilde{f}}(\la)}\mathcal{N}(\lambda)
  +\frac{4\eta_{\tilde{ f}}(\lambda)}{1-2\eta_{\tilde{f}}(\la)}.
\end{equation*}
by \eqref{nabla ineq}. Then thanks to \eqref{N upper bound},
\eqref{def eta}, \eqref{beta1}, \eqref{Sobolev ineq},
\eqref{mu-estimates},
and \eqref{wla bounded boundary} we conclude.
\end{proof}

The following proposition is  a doubling type result.
\begin{proposition}
There exists a constant $C_1>0$ such that for any $ \la \in (0,\frac{r_0}{2})$ and $T \in [1,2]$  
\begin{equation}\label{doub H}
	\frac{1}{C_1}H(T\la)\le H(\la) \le C_1 H(T \la),
\end{equation}
\begin{equation}\label{doub wla}
  \int_{B_T} |W^\la(y)|^2 dy \le 2^NC_1
  \int_{B_1} |W^{T\la}(y)|^2 \, dy,	
\end{equation}
and 
\begin{equation}\label{doub nabla wla}
  \int_{B_T \setminus \tilde{\Gamma}} |\nabla
  W^\la(y)|^2 dy
  \le 2^{N-2}C_1 \int_{B_1 \setminus \tilde{\Gamma}} |\nabla W^{T\la}(y)|^2 \, dy.	
\end{equation}
\end{proposition}
\begin{proof}
  From \eqref{N upper bound}, \eqref{N lower bound}, \eqref{E as H'},
  and \eqref{beta1} we deduce that there exist two constants
  $\kappa_1>0$ and $\kappa_2>0$ such that, for any $r \in (0,r_0)$,
\begin{equation*}
  -\frac{2}{r}	\le
  -\frac{2\eta_f(r)}{r}	\le \frac{H'(r)}{H(r)} \leq \frac{2\mathcal{N}(r)+\kappa_1}{r} \le \frac{\kappa_2}{r}. 
\end{equation*}
Then \eqref{doub H} follows from an integration in $(\la,T\la )$ of
the above inequality.  Furthermore from \eqref{doub H} we obtain that,
for any $ \la \in (0,\frac{r_0}{2})$ and $T \in [1,2]$,
\begin{align*}
	\int_{B_T} | W^\la(y)|^2 \, dy&= \frac{\la^{-N}}{H(\la)}\int_{ B_{\la T}} |U(y)|^2 \, dy \le
	\frac{C_1 2^N}{(\la T)^{N} H(T\la )}\int_{ B_{\la T}} |U(y)|^2 \, dy\\
	&= C_1 2^N \int_{ B_1 } | W^{T\la}(y)|^2 \, dy.
\end{align*}
In the same way \eqref{doub nabla wla} follows from \eqref{doub H}.
\end{proof}

\begin{proposition} \label{nablaboundarybounded}
 Let $\mathcal M$ be
    as in Proposition \ref{p:intparts} and $W^\la$ be defined in
  \eqref{wla}. Then there exist $M > 0$ and $\la_0 > 0$ such that for
  any $\la \in (0, \la_0)$ there exists $T_\la \in [1,2]$ such that
  $\lambda T_\lambda\not\in\mathcal M$ and
\begin{equation}\label{nablaboundaryboundedineq}
	\int_{\partial B_{T_\la}} |\nabla  W^\la|^2 \, dS \le M
        \int_{B_{T_\la} \setminus \tilde{\Gamma}} (|\nabla  W^\la|^2
        \, +| W^{\la}|^2)\,dy.
\end{equation}
\end{proposition}
\begin{proof}
  Since $\{ W^\la\}_{\la \in (0,r_0/2)}$ is bounded in
  $H^1(B_2 \setminus \tilde{\Gamma})$ by Proposition \ref{wla
    bounded}, \eqref{doub wla} and \eqref{doub nabla wla}, then
\begin{equation}\label{eq:32}
  \limsup_{\la \to 0^+}\int_{B_2 \setminus \tilde{\Gamma}}(|\nabla
  W^\la|^2 \, +| W^{\la}|^2)\, dy
  < +\infty.	
\end{equation}
By the Coarea formula, for any $\la \in (0,\frac{r_0}{2})$ the function
\[
  g_\la(r):=\int_{B_r\setminus\tilde\Gamma}(|\nabla  W^\la|^2 \, +| W^{\la}|^2) \, dy
\]
is absolutely continuous in $[1,2]$  with weak  derivative 
\begin{equation*}
	g'_\la(r)=\int_{\partial B_r}(|\nabla  W^\la|^2 \, +| W^{\la}|^2) \, dS \quad \text{for a.e. } r \in [1,2],
\end{equation*}
where the integral $\int_{\partial B_r}|\nabla
  W^\la|^2dS$ is meant in the sense of Remark \ref{rem:gradsense}.
To prove the statement we argue by contradiction.
If  the conclusion  does not hold, for any $M>0$ there
exists a sequence $\{\la_n\}_{n \in \mathbb{N}} \subset (0,r_0/2)$ such that
$\lim_{n\to \infty} \la_n=0$ and
\begin{equation*}
  \int_{\partial B_r}(|\nabla  W^{\la_n}|^2 \, +| W^{\la_n}|^2) \, dS
  >
  M \int_{B_r\setminus \tilde{\Gamma}}(|\nabla  W^{\la_n}|^2 \,
  +| W^{\la_n}|^2) \, dy
\end{equation*}
for any $n \in \mathbb{N}$ and $r\in [1,2]\setminus
\frac1{\lambda_n}\mathcal M$, and hence for
  a.e. $r\in[1,2]$. Hence 
\begin{equation*}
	g'_{\la_n}(r)> M g_{\la_n}(r) \quad \text{for any } n \in \mathbb{N} \text{ and a.e. } r \in [1,2].
\end{equation*}
An integration in $[1,2]$ yields 
\begin{equation*}
	\limsup_{n \to \infty} g_{\la_n}(1) \le e^{-M}\limsup_{n \to \infty} g_{\la_n}(2)
\end{equation*}
hence 
\begin{equation*}
	\liminf_{\la \to 0^+} g_{\la}(1) \le e^{-M}\limsup_{\la \to 0^+} g_{\la}(2).
\end{equation*}
In view of \eqref{eq:32}, letting $M \to \infty $ we conclude that 
\begin{equation*}
	\liminf_{\la \to 0^+}\int_{B_1 \setminus \tilde{\Gamma}}(|\nabla  W^\la|^2 \, +| W^{\la}|^2) \, dy =0.
\end{equation*}
 Then there exists a
sequence $\{\rho_{n}\}_{k \in \mathbb{N}}$ such that
$ W^{\rho_{n}} \to 0 $ strongly in
$H^1(B_1 \setminus \tilde{\Gamma})$ as $n\to\infty$.  Due to the
continuity of the trace operator $\gamma_1$ defined in Proposition
\ref{trace r} and \eqref{mu-O}, this is in
contradiction with \eqref{wla bounded boundary}.
\end{proof}

\begin{proposition}\label{wla nabla boundary bounded prop}
There exists  $\overline M>0$ such that 
\begin{equation}\label{wla nabla boundary bounded}
	\int_{\mathbb{S}^{N-1}} |\nabla	 W^{\la T_\la}|^2 \, dS \le \overline M \quad \text{for all } \la \in \left(0,\min\left\{\frac{r_0}{2}, \lambda_0\right\}\right).
\end{equation}
\end{proposition}
\begin{proof}
Since 
\begin{equation*} 
	\int_{\mathbb{S}^{N-1}} |\nabla	 W^{\la T_\la}|^2 \, dS =\frac{\la^2T_\la^{3-N}}{H(\la T_\la)}	\int_{\partial B_{T_\la}} |\nabla	U(\la y)|^2 \, dS = 
	T_\la^{3-N}\frac{H(\la)}{H(\la T_\la)}\int_{\partial B_{T_\la}} |\nabla  W^\la|^2 \, dS,
\end{equation*}
then, by \eqref{doub H}, \eqref{doub wla}, \eqref{doub nabla wla},
\eqref{nablaboundaryboundedineq}, and the fact that $1\leq T_\la \le 2$, 
for any $\la \in \left(0,\min\left\{\frac{r_0}{2},\lambda_0\right\}\right)$ we have that
\begin{align*} 
  \int_{\mathbb{S}^{N-1}} |\nabla	 W^{\la T_\la}|^2 \, dS
  &\le 2C_1 M \int_{B_{T_\la} \setminus \tilde{\Gamma}} (|\nabla  W^\la|^2+| W^\la|^2) \, dy\\ 
  &\le 2^{N+1}C^2_1 M\int_{B_1 \setminus \tilde{\Gamma}}(|\nabla  W^{T_\la\la}|^2+| W^{T_\la\la}|^2)\, dy.
\end{align*}
Therefore we  conclude thanks to Proposition \ref{wla bounded}.
\end{proof}

Thanks to the estimates established above, we can now
  prove a first blow-up result.
\begin{proposition} \label{u-asymptotic}
  Let  $u \in H^1(B_R\setminus \Gamma)$, $u\not\equiv0$,  be a non-trivial
  weak solution to
  \eqref{op},  with $\Gamma$ defined in  \eqref{Gamma
  2}--\eqref{Gamma} and $f$ satisfying either \eqref{h1} or
   \eqref{h2}, and let
  $U=u \circ F$ be the corresponding solution to
  \eqref{straightenedproblem}.
  Let $\gamma$ be as in \eqref{limit gamma }.
Then 
\begin{equation} \label{gamma as k0/2}
	\text{there exists } k_0 \in \mathbb{N} \text{ such that } \gamma=\frac{k_0}{2}.
\end{equation}
 For  any sequence $\{\la_n\}_{n \in \mathbb{N}}$ with $\lim_{n\to
    \infty}\la_n=0$ there exists a subsequence $\{\la_{n_k}\}_{k \in
    \mathbb{N}}$ and  an eigenfunction $\Psi$ of problem
  \eqref{sphereprob} associated to the eigenvalue $\mu_{k_0}$ such
  that $\norm{\Psi}_{L^2(\mathbb{S}^{N-1})}=1$ and
\begin{equation} \label{u asymptotic}
	\frac{U(\la_{n_k} y)}{\sqrt{H(\la_{n_k})}} \to |y|^\gamma \Psi\left(\frac{y}{|y|}\right) \quad \text{strongly in } H^1(B_1 \setminus \tilde \Gamma).
\end{equation}
\end{proposition}
\begin{proof}
  Let $ W^\la$ be as in \eqref{wla} for any
  $\la \in\left(0,\min\left\{\frac{r_0}{2}, \lambda_0\right\}\right)$
  and let us consider a sequence $\{\la_n\}_{n \in \mathbb{N}}$ such
  that $\lim_{n\to \infty}\la_n=0$. From Proposition \ref{wla bounded}
  $\{ W^{\la T_\la}:\la \in\left(0,\min\left\{\frac{r_0}{2},
      \lambda_0\right\}\right)\}$ is bounded in
  $H^1(B_1 \setminus \tilde{\Gamma})$.  Therefore there exists a
  subsequence
  $\{ W^{\la_{n_k} T_{\la_{n_k}}}\}_{k \in \mathbb{N}} \subset H^1(B_1
  \setminus \tilde{\Gamma})$ and a function
  $ W \in H^1(B_1 \setminus \tilde{\Gamma})$ such that
  $ W^{\la_{n_k} T_{\la_{n_k}}} \rightharpoonup W$ weakly in
  $H^1(B_1 \setminus \tilde{\Gamma})$.  By compactness of the trace
  operator $\gamma_1$ (see Proposition \ref{trace r}),
 \eqref{mu-O}, and \eqref{wla bounded boundary}, it
  follows that 
  \begin{equation}\label{eq:41}
    \int_{\partial B_1} W^2\, dS =1
  \end{equation}
and so $W\not\equiv 0$ on $B_1 \setminus \tilde{\Gamma}$.

By H\"older's inequality and \eqref{Found ineq} we have that, for
every $\phi\in H^1(B_1\setminus\tilde\Gamma)$,
\begin{multline}\label{eq:33}
	\left|\la^2\int_{B_1}\tilde f(\la y ) W^{\la}(y) \phi(y) \,
          dy\right|	\\
 \le 
	\la^2\eta_{\tilde f(\la\cdot)}(1) 
	 \left( \int_{B_1 \setminus \tilde\Gamma}|\nabla W^{\la}|^2 \, dy+
	\int_{\partial B_1} |W^{\la}|^2 \, dS\right)^{\frac{1}{2}} \left( \int_{B_1 \setminus \tilde\Gamma}|\nabla \phi|^2 \, dy+
	\int_{\partial B_1} \phi^2 \, dS\right)^{\frac{1}{2}}.
    \end{multline}
By \eqref{def eta}  and a change of variables we have
that
\begin{align}\label{eq:34}
	\la^2\eta_{\tilde f(\la\cdot)}(1)&=S_{N,q_\epsilon} \la^2
	\left(\int_{B_1}|\tilde f(\la y)|^{\frac{N}{2}+\epsilon} \,
          dy\right)^\frac{2}{N+2\epsilon} \\
\notag&=S_{N,q_\epsilon} \la^\frac{4 \epsilon }{N+2
                \epsilon}\|\tilde
                f\|_{L^{\frac{N}{2}+\epsilon}(B_\lambda)}\to
                0\quad\text{as }\lambda\to0^+.
              \end{align}
From \eqref{eq:33}, \eqref{eq:34}, the boundedness of
  $\{W^\lambda\}$ in $H^1(B_1\setminus\tilde\Gamma)$ (established in
  Proposition \ref{wla bounded}) and of the traces (following from
  Proposition \ref{trace r}), we deduce that
\begin{equation}\label{eq:37}
\lim_{k \to \infty} \lambda_{\la_{n_k}
  T_{\la_{n_k}}}^2\int_{B_{1}}\tilde f({\la_{n_k} T_{\la_{n_k}}} y )
W^{\la_{n_k} T_{\la_{n_k}}}(y) \phi(y) \, dy=0,
\end{equation}
 for
every $\phi\in H^1(B_1\setminus\tilde\Gamma)$.

Let $\phi \in H_{0,\partial B_1}^1(B_1 \setminus \tilde{\Gamma})$.
We can test \eqref{wla equa1} with $\phi$  to obtain
\begin{multline}\label{eq:35}
\int_{B_{1} \setminus \tilde \Gamma}A(\la_{n_k}
  T_{\la_{n_k}} y)\nabla W^{\la_{n_k}
  T_{\la_{n_k}}}(y) \cdot \nabla \phi(y) \, dy \\=(\la_{n_k}
T_{\la_{n_k}})^2\int_{B_{1}}\tilde f({\la_{n_k} T_{\la_{n_k}}} y )
W^{\la_{n_k} T_{\la_{n_k}}}(y) \phi(y) \, dy,
\end{multline}
for any $k \in \mathbb{N}$.  Since
$ W^{\la_{n_k} T_{\la_{n_k}}} \rightharpoonup W$ weakly in
$H^1(B_1 \setminus \tilde{\Gamma})$, by \eqref{A-O} we have
that
\begin{equation}\label{eq:36}
\lim_{k\to \infty}
\int_{B_{1} \setminus \tilde \Gamma}A(\la_{n_k}
  T_{\la_{n_k}} y)\nabla W^{\la_{n_k}
  T_{\la_{n_k}}}(y) \cdot \nabla \phi(y) \, dy=\int_{B_{1} \setminus
  \tilde{\Gamma}}\nabla  W \cdot \nabla \phi\, dy.
\end{equation}
Therefore, for any $\phi \in H^1_{0,\partial B_1}(B_1\setminus
\tilde{\Gamma})$ we can pass to the limit as $k\to\infty$
  in \eqref{eq:35} thus obtaining, in view of \eqref{eq:36} and
  \eqref{eq:37},
\[
\int_{B_1\setminus \tilde{\Gamma}}  \nabla   W \cdot \nabla \phi \,
dy=0,
\] 
i.e. $ W$ is a weak solution of 
\begin{equation}\label{wproblem}
	\begin{cases}
          -\Delta W =0, &\text{ on } B_1\setminus \tilde{\Gamma},\\
          \dfrac{\partial^+W}{\partial\nu^+}=\dfrac{\partial^-W}{\partial\nu^-}=0,
          & \text{ on } \tilde{\Gamma}.
	\end{cases}
\end{equation}
We note that, by classical elliptic regularity theory, $W$ is smooth in $B_1\setminus \tilde{\Gamma}$.

In view of \eqref{wla} and Propositions \ref{nablaboundarybounded} and
\ref{p:intparts}, by scaling  we have that, for every $\phi \in H^1(B_1\setminus
\tilde{\Gamma})$,   
\begin{multline} \label{aproxequationboundry}
	\int_{B_{1} \setminus \tilde\Gamma} A(\la_{n_k}
          T_{\la_{n_k}}y)
\nabla W^{\la_{n_k} T_{\la_{n_k}}}(y) \cdot \nabla \phi(y) \, dy  \\
	-(\la_{n_k} T_{\la_{n_k}})^2\int_{B_{1}}\tilde
                f({\la_{n_k} T_{\la_{n_k}}} y ) W^{\la_{n_k}
                T_{\la_{n_k}}}(y) \phi(y) \, dy \\=\int_{\partial B_1} 
(A(\la_{n_k}
          T_{\la_{n_k}}y)
\nabla W^{\la_{n_k} T_{\la_{n_k}}}(y) \cdot\nu)\,  \phi(y) \, dS.
\end{multline}
Thanks to Proposition \ref{wla nabla boundary bounded prop}
and \eqref{A-oper-norm}
there exists a function $h \in L^2\partial B_1)$ such that
\begin{equation}\label{eq:38}
  (A(\la_{n_k}
          T_{\la_{n_k}}y)
          \nabla W^{\la_{n_k} T_{\la_{n_k}}}(y)
          \cdot\nu)\rightharpoonup h
        \quad\text{weakly in $L^2(\partial B_1)$,}
      \end{equation}
      up to a subsequence.
By the weak convergence 
$ W^{\la_{n_k} T_{\la_{n_k}}} \rightharpoonup W$ in
$H^1(B_1 \setminus \tilde{\Gamma})$, \eqref{A-O}, \eqref{eq:37}, and \eqref{eq:38}, passing to the limit
as $k \to \infty$ in \eqref{aproxequationboundry}, we obtain that 
\begin{equation}\label{eq:39}
  \int_{B_1 \setminus \tilde \Gamma} \nabla  W \cdot \nabla \phi \,
  dy = \int_{\partial B_1} h \phi \, dS
\end{equation}
for any  $\phi \in H^1(B_1\setminus \tilde{\Gamma})$.
From the compactness of the trace operator $\gamma_1$ (see
  Proposition \ref{trace r}) and \eqref{eq:38} it follows that
  \[
    \lim_{k \to \infty}
    \int_{\partial B_1} 
(A(\la_{n_k}
          T_{\la_{n_k}}y)
\nabla W^{\la_{n_k} T_{\la_{n_k}}}(y) \cdot\nu)\,  W^{\la_{n_k} T_{\la_{n_k}}}(y) \, dS
=\int_{\partial B_1} h W \, dS.
\]
Therefore, recalling estimates \eqref{eq:33}, \eqref{eq:34}, and the boundedness of
  $\{W^\lambda\}$ in $H^1(B_1\setminus\tilde\Gamma)$, choosing $\phi=
  W^{\la_{n_k} T_{\la_{n_k}}}$ in \eqref{aproxequationboundry} and
passing to the limit as $k \to \infty$,
we obtain that 
\begin{equation}\label{eq:40}
  \lim_{k \to \infty}\int_{B_1 \setminus  \tilde\Gamma} A(\la_{n_k}
          T_{\la_{n_k}}y)\nabla
  W^{\la_{n_k} T_{\la_{n_k}}}\cdot \nabla  W^{\la_{n_k}
    T_{\la_{n_k}}}\, dy = \int_{\partial B_1} h W \, dS.
\end{equation}
From \eqref{eq:39} and \eqref{eq:40} it follows that 
\[
  \lim_{k \to \infty}\int_{B_1 \setminus  \tilde\Gamma}  A(\la_{n_k}
          T_{\la_{n_k}}y)\nabla  W^{\la_{n_k} T_{\la_{n_k}}}\cdot \nabla  W^{\la_{n_k} T_{\la_{n_k}}} \, dy=
        \int_{B_1 \setminus \tilde \Gamma} |\nabla  W|^2\, dy
      \]
and so, thanks to \eqref{A-O},
\begin{equation}\label{wstrong}
	 W^{\la_{n_k} T_{\la_{n_k}}}\to  W \quad \text{strongly in } H^1(B_1 \setminus \tilde \Gamma). 
\end{equation}
For any $k \in \mathbb{N}$ and $r \in (0,1)$ let us define
\begin{align*}
  &E_k(r):=r^{2-N}\int_{B_{r} \setminus \tilde\Gamma}\!(  A(\la_{n_k}
  T_{\la_{n_k}}y)\nabla  W^{\la_{n_k} T_{\la_{n_k}}}\!\cdot\!
    \nabla  W^{\la_{n_k} T_{\la_{n_k}}}\\
  &\qquad\qquad\qquad-(\la_{n_k} T_{\la_{n_k}})^2 
  \tilde f({\la_{n_k} T_{\la_{n_k}}} y )| W^{\la_{n_k} T_{\la_{n_k}}}|^2 )\, dy,\\
  &H_k(r):=r^{1-N}\int_{\partial B_r}\mu({\la_{n_k}
    T_{\la_{n_k}}} y) | W^{\la_{n_k} T_{\la_{n_k}}}|^2 \, dS,
    \quad \text{and} \quad \mathcal{N}_k(r):=\frac{E_k(r)}{H_k(r)}.
\end{align*}
By a change of variables it is easy to verify that, for any $r \in (0,1)$,
\begin{equation}\label{eq:42}
  \mathcal{N}_k(r)=\frac{E_k(r)}{H_k(r)}=\frac{E(\la_{n_k}
    T_{\la_{n_k}}r)}{H(\la_{n_k}
    T_{\la_{n_k}}r)}=\mathcal{N}(\la_{n_k} T_{\la_{n_k}}r).
\end{equation}
For any $r \in (0,1)$, we also define 
\[
  H_{ W}(r):=r^{1-N}\int_{\partial B_r} | W|^2 \, dS,
  \quad
  E_{ W}(r):=r^{2-N}\int_{B_{r} \setminus \tilde \Gamma}
  |\nabla  W|^2 \, dy\quad \text{and} \quad \mathcal{N}_{
    W}(r):=\frac{E_{ W}(r)}{H_{ W}(r)}.
\]
The definition of $\mathcal{N}_{ W}$ is well posed. Indeed, if
$H_{ W}(r)=0$ for some $r \in (0,1)$, then we may test the equation
 \eqref{wproblem} on $B_r$ with $W$ and conclude that $W=0$ in
 $B_r$. Thanks to classical unique continuation principles
for harmonic functions,
this would imply that $W=0$ in $B_1$, thus contradicting \eqref{eq:41}.

Thanks to \eqref{wstrong},
\eqref{eq:33}-\eqref{eq:34} together with the boundedness of
  $\{W^\lambda\}$ in $H^1(B_1\setminus\tilde\Gamma)$, \eqref{A-O},
  \eqref{mu-O},
and Proposition \ref{existence limit gamma}, passing to the limit as
$k \to \infty$ in \eqref{eq:42} we obtain that
\begin{equation}\label{Nw}
	\mathcal{N}_{ W}(r)=\lim_{k \to \infty}\mathcal{N}_k(r)=\lim_{k \to \infty} \mathcal{N}(\la_{n_k} T_{\la_{n_k}}r)=\gamma \quad \text{for any }r \in (0,1).
\end{equation}
Then $\mathcal{N}_{ W}$ is constant in $(0,1)$. 
Following the proof of  Proposition \ref{N'} in the case $f \equiv 0$ and $g \equiv 0$ (where $g$ is the function defined
in \eqref{condition_g1}--\eqref{condition_g2}), so that $A=\mathop{\rm Id}_{N}$ 
and $\mu=1$, we obtain that
\[
0=\mathcal{N}_W'(r)\ge \frac{2r\left(\left(
\int_{\partial B_r}\left|\pd{ W}{\nu}\right|^2 \, dS\right)\left(\int_{\partial B_r}{ W^2}dS\right)-
	\left(\int_{\partial B_r} W\pd{ W}{\nu} \,
          dS\right)^2\right)}{\left(\int_{\partial B_r} W^2\,
        dS\right)^2} \ge 0
\]  
for a.e.  $r \in (0,1)$.
It follows that $\left(\int_{\partial B_r}\left|\pd{ W}{\nu}\right|^2
  \, dS\right)
\left(\int_{\partial B_r}{ W^2}dS\right)=\left(\int_{\partial B_r}
  W\pd{w}{\nu} \, dS\right)^2$ for a.e. $r \in (0,1)$, i.e.  equality holds
in the Cauchy-Schwartz inequality  for the vectors $W$ and $\pd{ W}{\nu}$ in $L^2(\partial B_r)$ for  a.e. $r \in (0,1)$. It follows that there exists a function 
$\zeta(r)$ such that 
\begin{equation}\label{weqautionrtheta}
	\pd{ W}{\nu}(r\theta)=\zeta(r)  W(r\theta) \quad \text{for any }\theta \in \mathbb{S}^{N-1} \setminus \Sigma \text{ and a.e. } r \in (0,1].
\end{equation}
Multiplying by $ W(r \theta)$ and integrating on $\mathbb{S}^{N-1}$ we obtain
\[\int_{\mathbb{S}^{N-1}}\pd{ W}{\nu}(\theta r) W(r \theta) \, dS=\zeta(r) \int_{\mathbb{S}^{N-1}}  W^2(\theta r) \, dS,\]
so that $\zeta(r)=\frac{H_{ W}'(r)}{2H_{ W}(r)}=\frac{\gamma}{r} $ by
Proposition \ref{p:derH} and \eqref{Nw}.  Integrating \eqref{weqautionrtheta}
between $r \in (0,1)$ and $1$ we
obtain that
\begin{equation*}
	 W(r\theta)=r^{\gamma}  W(1\theta)=r^{\gamma} \Psi(\theta) \quad \quad \text{for any }\theta \in \mathbb{S}^{N-1} \setminus \Sigma \text{ and any } r \in (0,1],
\end{equation*}
where $\Psi= W|_{\mathbb{S}^{N-1} \setminus \Sigma}$. Then
$\Psi \in H^1(\mathbb{S}^{N-1} \setminus \Sigma)$;
 furthermore, substituting $W(r\theta)=r^{\gamma}
  \Psi(\theta)$ in \eqref{wproblem} we find out that 
$\Psi$ is an eigenfunction of
\eqref{sphereprob}  with $(\gamma+N-2)\gamma$ as an
  associated eigenvalue. Hence by Proposition \ref{eigen} there exists
$k_0 \in \mathbb{N}$ such that
$(\gamma+N-2)\gamma=\frac{k_0(k_0+2N-4)}{4}$. Recalling
  from Proposition \ref{existence limit gamma} that $\gamma\geq0$, we
  then obtain \eqref{gamma as k0/2}.

To conclude the proof it is enough to show that $ W^{\la_{n_k}} \to W$
strongly in $H^1(B_1 \setminus \tilde \Gamma)$ (possibly
  along a subsequence). Since
$\{ W^{\la_{n_k}}\}_{k \in \mathbb{N}}$ is bounded in $H^1(B_1 \setminus \tilde \Gamma)$ by Proposition
\ref{wla bounded}, there exists a function
$\tilde{W} \in H^1(B_1 \setminus \tilde \Gamma)$ and $T \in [1,2]$
such that $ W^{\la_{n_k}} \rightharpoonup \tilde W$ weakly in
$H^1(B_1 \setminus \tilde \Gamma)$ and $T_{\la_k}\to T$, up to a
subsequence.  

Moreover, since 
$\{ W^{\la_{n_k} T_{\la_{n_k}}}\}_{k \in \mathbb{N}}$
and $\{|\nabla W^{\la_{n_k} T_{\la_{n_k}}}|\}_{k \in \mathbb{N}}$ converge strongly in
$L^2(B_1)$ by \eqref{wstrong}, they are dominated by a measurable $L^2(B_1)$-function,
up to a subsequence.  Similarly, thanks to \eqref{doub H}, we can
suppose that, up to a subsequence, the limit
\[\zeta:=\lim_{k \to \infty}\frac{H(\la_{n_k}T_{\lambda_{n_k}})}{H(\la_{n_k})}\] 
exists and it is finite and strictly positive. Then for
any $\phi \in C^{\infty}_c(B_1)$ we have that
\begin{align*}
  &\lim_{k \to \infty}\int_{B_1}  W^{\la_{n_k}}(y) \phi(y) \, dy=
\lim_{k \to \infty} T^N_{\la_{n_k}} \int_{B_{T^{-1}_{\la_{n_k}}}}
    W^{\la_{n_k}}(T_{\la_{n_k}}y)
 \phi(T_{\la_{n_k}}y) \, dy\\
  &=\lim_{k \to \infty} T^N_{\la_{n_k}}
\sqrt{\frac{H(\la_{n_k}T_{\la_{n_k}})}{H(\la_{n_k})}}
\int_{B_{T^{-1}_{\la_{n_k}}}} W^{T_{\la_{n_k}}\la_{n_k}}(y) \phi(T_{\la_{n_k}}y) \, dy\\
  &=T^N \sqrt{\zeta}\int_{B_{T^{-1}}}  W(y) \phi(Ty) \, dy=
 \sqrt{\zeta}\int_{B_1} W(y/T) \phi(y) \, dy,
\end{align*}
thanks to the Dominated Convergence Theorem. By density the same
holds for any $\phi \in L^2(B_1)$. It follows that
$W^{\la_{n_k}} \rightharpoonup \sqrt{\zeta}W(\cdot/T)$ weakly in
$L^2(B_1)$. Hence, by uniqueness of the weak limit, we have that
$\tilde W(\cdot)=\sqrt{\zeta} W(\cdot/T)$ and
$ W^{\la_{n_k}} \rightharpoonup \sqrt{\zeta} W(\cdot/T)$ weakly in
$H^1(B_1 \setminus \tilde \Gamma)$. Furthermore
\begin{align*}
  &\lim_{k \to \infty}\int_{B_1\setminus \tilde \Gamma} |\nabla  W^{\la_{n_k}}(y)|^2\,
    dy=
    \lim_{k \to \infty} T^N_{\la_{n_k}}
    \int_{B_{T^{-1}_{\la_{n_k}}}\setminus \tilde \Gamma}|\nabla 
    W^{\la_{n_k}}(T_{\la_{n_k}}y)|^2 \, dy\\
  &=\lim_{k \to \infty}
    T _{\la_{n_k}}^{N{-2}}\frac{H(\la_{n_k}T_{\la_{n_k}})}
    {H(\la_{n_k})}\int_{B_{T^{-1}_{\la_{n_k}}}\setminus \tilde \Gamma}|\nabla
    W^{T_{\la_{n_k}}\la_{n_k}}(y)|^2 
    dy\\
  &=T^{N-2} \zeta\int_{B_{T^{-1}}\setminus
    \tilde \Gamma} |\nabla W(y)|^2 dy
=\int_{B_1\setminus \tilde \Gamma} | 
    \sqrt{\zeta}\nabla  (W(\cdot/T))|^2\, dy.
\end{align*}
Then we can conclude that
$W^{\la_{n_k}} \to \tilde W=\sqrt{\zeta}W(\cdot/T)$ strongly in
$H^1(B_1 \setminus\tilde \Gamma)$.
Moreover, 
 by compactness of the trace
  operator $\gamma_1$ (see Proposition \ref{trace r}),
  \eqref{mu-O}, and \eqref{wla bounded boundary},
we deduce that
$\int_{\partial B_1}\tilde W^2 \, dS=1$. Then, since
$ W(r \theta)=r^\frac{k_0}{2}\Psi(\theta)$, we deduce that
\[
\tilde W(r \theta)=\sqrt{\zeta} W\left(\frac{r}{T}\theta\right)=\left(\frac{\zeta}{T^{k_0}}\right)^{\frac{1}{2}} {r}^{\frac{k_0}{2}}\Psi(\theta)=\left(\frac{\zeta}{T^{k_0}}\right)^{\frac{1}{2}}  W(r \theta)\]
and 
\[
1=\int_{\partial B_1}\tilde W^2 \,
dS=\frac{\zeta}{T^{k_0}}\int_{\partial B_1} W^2 \,
dS=\frac{\zeta}{T^{k_0}},
\]
thanks to \eqref{eq:41}.
Therefore $ W=\tilde{W}$ and  the proof is complete.
\end{proof}

We are now in position of prove Theorem \ref{t:ucp-crack}.
\begin{proof}[Proof of Theorem \ref{t:ucp-crack}]
  Let us assume that   $\mathop{\rm Tr}^+_\Gamma u(z)=O(|z|^k)$ as $|z| \to 0^+$,
  $z\in\Gamma$, for all $k \in \mathbb{N}$ (a similar argument works
  under the assumption  $\mathop{\rm Tr}^-_\Gamma
  u(z)=O(|z|^k)$). Letting $U=u\circ F$, by the properties of the
  diffeomorphism $F$ described in Proposition \ref{diffeomorphism}, we
  have that $\mathop{\rm Tr}^+_{\tilde \Gamma} U(z)=O(|z|^k)$ as $|z|
  \to 0^+$, so that, for all $k\in \mathbb{N}$,
  \begin{equation}\label{eq:14}
    \|\lambda^{-k}\mathop{\rm Tr}\nolimits^+_{\tilde
      \Gamma}U(\lambda\cdot)\|_{L^2(B_1\cap\tilde\Gamma)}\to
    0\quad\text{as }\lambda\to0^+.
  \end{equation}
On the other hand, if, by contradiction,  $u\not\equiv0$, by Proposition
\ref{u-asymptotic} and classical trace theorems there exist
$k_0\in{\mathbb N}$, a sequence $\lambda_n\to 0^+$, and 
  an eigenfunction $\Psi$ of problem
  \eqref{sphereprob} such that
  \begin{equation}\label{eq:17}
    \lim_{n\to\infty}	\frac{\|\mathop{\rm Tr}\nolimits^+_{\tilde
      \Gamma}U(\la_{n}
    \cdot)\|_{L^2(B_1\cap\tilde\Gamma)}}{\sqrt{H(\la_{n})}} =
 \left \|\mathop{\rm Tr}\nolimits^+_{\tilde
      \Gamma}\left(|y|^\gamma
    \Psi\left(\tfrac{y}{|y|}\right)\right)\right\|_{L^2(B_1\cap\tilde\Gamma)}\neq0,
\end{equation}
where the above limit is nonzero thanks to Remark \ref{rem:eigenfunctionnonzero}. 
Combining \eqref{eq:14} and \eqref{eq:17} we obtain that
\[
  \lim_{n\to\infty}
  \frac{\sqrt{H(\la_{n})}}{\lambda_n^k}=0\quad\text{for all
  }k\in\mathbb N,
  \]
thus contradicting estimate \eqref{H lower bound}.
\end{proof}

\section{Asymptotics of the height function $H(\lambda)$ as
  $\lambda\to0^+$, when $N\geq3$}
\label{sec:colorr-height-funct}

In dimension $N\geq 3$, we can further specify the
   behaviour of $U(\lambda\cdot)$ as $\lambda\to0^+$, deriving
  the asymptotics of the function  $H(\lambda)$ appearing as a
  normalization factor in the blowed-up family \eqref{wla}. 
Let $\{Y_{k,i}\}_{k\in\mathbb{N},
  i=1,\dots,N_k}$ be the basis of $L^2(\mathbb{S}^{N-1})$ given by Proposition
\ref{eigen}.
Let  $N\geq3$, $u \in H^1(B_R\setminus \Gamma)$  be a 
  weak solution to
  \eqref{op},  with $\Gamma$ defined in  \eqref{Gamma
  2}--\eqref{Gamma} and $f$ satisfying either \eqref{h1} or
   \eqref{h2}, and let
  $U=u \circ F$ be the corresponding solution to
  \eqref{straightenedproblem}.
  For any  $\la \in (0,r_0)$, $k \in \mathbb{N}$ and $i=1,\dots, N_k$
  we define
\begin{equation}\label{Fourier-coefficents}
\varphi_{k,i}(\la):=\int_{\mathbb{S}^{N-1}}U(\la\theta ) Y_{k,i}(\theta) \, dS	
\end{equation}
and 
\begin{align}\label{resto}
  \Upsilon_{k,i}(\la):=&-\int_{B_\la\setminus \tilde
                         \Gamma}(A-\mathop{\rm Id}\nolimits_N) \nabla U
                         \cdot
                         \frac{\nabla_{\mathbb{S}^{N-1}}Y_{k,i}(y/|y|)}{|y|}\, dy  \\
                       &+\int_{B_\la} \tilde{f}(y)U(y)Y_{k,i}(y/|y|)
                         \, dy+
                         \int_{\partial B_\la}(A-\mathop{\rm Id}\nolimits_N)
                         \nabla U \cdot \frac{y}{|y|}Y_{k,i}(y/|y|) \, dS. \notag 
\end{align}

\begin{proposition}
  Let $k_0$ be as in Proposition \ref{u-asymptotic}. Then, for any
  $i=1,\dots, N_{k_0}$ and $r \in (0,r_0]$,
\begin{align}\label{Fourier-coefficents-expansion}
\varphi_{k_0,i}(\la)&=\la^{\frac{k_0}{2}}\Big(r^{-\frac{k_0}{2}}\varphi_{k_0,i}(r)+\frac{2N+k_0-4}{2(N+k_0-2)}\int_{\la}^r s^{-N+1-\frac{k_0}{2}}\Upsilon_{k_0,i}(s) \, ds\\
&+\frac{k_0r^{-N+2-k_0}}{2(N+k_0-2)} \int_{0}^rs^{\frac{k_0}{2}-1} \Upsilon_{k_0,i}(s) \, ds\Big) +o(\la^{\frac{k_0}{2}})  \quad \text{as } \la \to 0^+.\notag
\end{align}
\end{proposition}
\begin{proof}
  For any $k \in \mathbb{N}$ and any $i=1,\dots, N_k$ we consider the
  distribution $\zeta_{k,i}$ on $(0,r_0)$ defined as
\begin{align*}
\sideset{_{\mathcal{D}'(0,r_0)}}{_{\mathcal{D}(0,r_0)}}{\mathop{\langle
  \zeta_{k,i}, \omega
  \rangle}}&:=\int_0^{r_0} \omega(\la) 
\left(\int_{\mathbb{S}^{N-1}} \tilde f(\la \theta) U(\la \theta)
             Y_{m,k}(\theta) \, dS_{\theta}\right)
             \, d\la \\ 
&\qquad-\int_{B_{r_0} \setminus \tilde \Gamma} (A-\mathop{\rm
                           Id}\nolimits_N)
                           \nabla U \cdot \nabla\Big(|y|^{1-N}\omega(|y|) Y_{m,k}(y/|y|)\Big) \, dy,
\end{align*}
for any $\omega \in \mathcal{D}(0,r_0)$.

Since $\Upsilon_{k,i} \in L^1_{\rm loc}(0,r_0)$ by \eqref{resto},
we may  consider  its derivative in the sense of distributions. A direct calculation shows that
\begin{equation}\label{eq:7}
	\Upsilon_{k,i}'(\la)=\la^{N-1}\zeta_{k,i}(\la)
\end{equation}
in the sense of distributions on $(0,r_0)$.  From the definition of
$\zeta_{k,i}$, \eqref{straightenedproblem}, and the fact that $Y_{k,i}$
is a solution of \eqref{weak-eig} we deduce that
\[
-\varphi_{k,i}''(\la)-\frac{N-1}{\la}\varphi'_{k,i}(\la)+
\frac{\mu_k}{\la^2}\varphi_{k,i}(\la)= \zeta_{k,i}(\la)
\]
in the sense of distribution in $(0,r_0)$; the above equation can be
rewritten as   
\[
-(\la^{N-1+k}(\la^{-\frac{k}{2}}\varphi_{k,i}(\la))')'=
{\la^{N-1+\frac k2}}
\zeta_{k,i}(\la),
\]
thanks to \eqref{muk}.  Integrating the right-hand side of the
equation above by parts, since \eqref{eq:7} holds, we obtain that, for
every $r \in (0,r_0)$, $k \in \mathbb{N} $ and 
$i=1,\dots, N_{k}$ there exists a constant $c_{k,i}(r)$ such that
\[
(\la^{-\frac{k}{2}}\varphi_{k,i}(\la))'=-\la^{-N+1-\frac{k}{2}}\Upsilon_{k,i}(\la)-\frac{k}{2}\la^{-N+1-k}\left(c_{k,i}(r)+\int_{\la}^r
  s^{\frac{k}{2}-1}\Upsilon_{k,i}(s) \, ds\right)\]
in the sense of distribution on $(0,r_0)$. Then
$\varphi_{k,i}(\la) \in W^{1,1}_{\rm loc}(0,r_0)$ and a further
integration yields
 \begin{align}\label{eq:13}
\varphi_{k,i}(\la)&=\la^{\frac{k}{2}}\Bigg(r^{-\frac{k}{2}}\varphi_{k,i}(r)+\int_\la^r s^{-N+1-\frac{k}{2}}\Upsilon_{k,i}(s) \, ds\Bigg)\\
 \notag           &+\frac{k}{2}\la^{\frac{k}{2}}\Bigg(\int_\la^r s^{-N+1-k}\left(c_{k,i}(r)+\int_{s}^r t^{\frac{k}{2}-1}\Upsilon_{k,i}(t) \, dt\right) \, ds \\
 \notag           &=\la^{\frac{k}{2}}\left(r^{-\frac{k}{2}}\varphi_{k,i}(r)+\frac{2N+k-4}{2(N+k-2)}\int_\la^r s^{-N+1-\frac{k}{2}}\Upsilon_{k,i}(s) \, ds\right)\\
 \notag           &-\la^{\frac{k}{2}}\frac{kc_{k,i}(r)r^{-N+2-k}}{2(N+k-2)}+\frac{k \la^{-N+2-\frac{k}{2}}}{2(N+k-2)} \left(c_{k,i}(r)+\int_{\la}^r t^{\frac{k}{2}-1}\Upsilon_{k,i}(t) \, dt\right).
\end{align}
Now we claim that, if $k_0$ is as in Proposition \ref{u-asymptotic}, then 
\begin{equation}\label{eq:8}
\text{the function } s \to  s^{-N+1-\frac{k_0}{2}}\Upsilon_{k_0,i}(s) \text{ belongs to } L^1(0,r_0).
\end{equation}
To this end we will estimate each terms in \eqref{resto}.  Thanks to
\eqref{A-O}, H\"older's inequality, a change of
variables and Proposition \ref{wla bounded}, we have that 
\begin{align*}
  &\left|\int_{B_s\setminus \tilde \Gamma}(A-\mathop{\rm Id}\nolimits_N) \nabla U \cdot \frac{\nabla_{\mathbb{S}^{N-1}}Y_{k_0,i}(y/|y|)}{|y|} dy\right| 
\le \co \int_{B_s\setminus \tilde \Gamma}|y||\nabla U|\frac{|\nabla_{\mathbb{S}^{N-1}}Y_{k_0,i}(y/|y|)|}{|y|} \, dy\\
&\le \co \left(\int_{B_s\setminus \tilde \Gamma}|\nabla U|^2 dy \right)^{\frac12}\left(\int_{B_s\setminus \tilde \Gamma}|\nabla_{\mathbb{S}^{N-1}}Y_{k_0,i}(y/|y|)|^2 dy \right)^{\frac12}\\
&\le \co \,s^{\frac{N-2}{2}} s^{\frac{N}{2}}
  \sqrt{H(s)}\left(\int_{B_1\setminus \tilde \Gamma}
|\nabla  W^s(y)|^2\ dy \right)^{\frac12} \le \co \,s^{N-1}\sqrt{H(s)}.
\end{align*}
From H\"older's inequality, \eqref{Found ineq}, 
  \eqref{mu-estimates},  and Proposition \ref{wla bounded}  it follows that
\begin{align*}
&\left|\int_{B_s} \tilde{f}(y)U(y)Y_{k_0,i}(y/|y|) \, dy\right| 
\le \left(\int_{B_s} |\tilde{f}(y)|U^2(y) \, dy\right)^{\frac12}
\left(\int_{B_s} |\tilde{f}(y)|Y_{k_0,i}^2(y/|y|) \, dy\right)^{\frac12}\\
&\le \co \, s^{\frac{4\epsilon}{N+2 \epsilon}}\left(\int_{B_s \setminus \tilde \Gamma}|\nabla U|^2 \, dy+s^{N-2}H(s)\right)^{\frac{1}{2}} \left(\int_{B_s \setminus \tilde \Gamma}|\nabla Y_{k_0,i}(y/|y|)|^2 \, dy+s^{N-2}\right)^{\frac{1}{2}} \\
&\le \co \, s^{(N-2)+\frac{4\epsilon}{N+2 \epsilon}} \sqrt{H(s)}.
\end{align*}
Furthermore, in view of \eqref{A-O}, for a.e. $s \in
(0,r_0)$ we have that 
\begin{equation*}
  \left|\int_{\partial B_s}(A-\mathop{\rm Id}\nolimits_N) 
    \nabla U \cdot \frac{y}{|y|}Y_{k_0,i}(y/|y|) \, dS\right|
  \le \co \, s\int_{\partial B_s} |\nabla U||Y_{k_0,i}(y/|y|)| \, dS
\end{equation*}
and an integration by parts and H\"older's inequality yield, for any $r \in (0,r_0]$,
\begin{align*}
&\int_0^{r} s^{-N+2-\frac{k_0}{2}}\left(\int_{\partial B_s} |\nabla U||Y_{k_0,i}(y/|y|)| \, dS\right) \, ds=r^{-N+2-\frac{k_0}{2}}\int_{B_{r}\setminus \tilde \Gamma}|\nabla U||Y_{k_0,i}(y/|y|)|\\
&+\left(N-2+\frac{k_0}{2}\right)\int_0^{r} s^{-N+1-\frac{k_0}{2}}\left(\int_{B_s\setminus \tilde \Gamma} |\nabla U||Y_{k_0,i}(y/|y|)| \, dS\right) \, ds \\
&\le \co \left(r^{1-\frac{k_0}{2}} \sqrt{H(r)}+\int_0^{r}s^{-\frac{k_0}{2}} \sqrt{H(s)} \, ds\right),
\end{align*}
reasoning as above.
In conclusion, combining the above estimates with \eqref{gamma as
  k0/2} and \eqref{H upper bound}, we obtain that, for any $r \in (0,r_0]$,
\begin{align}\label{eq:19}
  \int_0^{r} s^{-N+1-\frac{k_0}{2}}|\Upsilon_{k_0,i}(s)| \, ds &\le 
	\co \left(r^{1-\frac{k_0}{2}} \sqrt{H(r)}+\int_0^{r} s^{-\frac{k_0}{2}-1+
    \frac{4\epsilon}{N+2 \epsilon}} \sqrt{H(s)} \, ds \right)\\
  &\le \co \left(r+\int_0^{r}
    s^{\frac{2\epsilon-N}{N+2\epsilon}}\, ds \right)
\leq \co \, r^{\frac{4\epsilon}{N+2\epsilon}}\notag 
\end{align}
which in particular implies \eqref{eq:8}.
By \eqref{eq:8}, it follows that, for every $r \in (0,r_0]$,
\begin{align}\label{eq:10}
&\la^{\frac{k_0}{2}}\left(r^{-\frac{k_0}{2}}\varphi_{k_0,i}(r)+\frac{2N+k_0-4}{2(N+k_0-2)}\int_\la^r
  s^{-N+1-\frac{k_0}{2}}\Upsilon_{k_0,i}(s) \,
  ds-\frac{k_0c_{k_0,i}(r)r^{-N+2-k_0}}{2(N
+k_0-2)}\right)\\ 
&= O\left(\la^{\frac{k_0}{2}}\right)=	o\left(\la^{-N+2-\frac{k_0}{2}}\right) \quad \text{as } \la \to 0^+ \notag 
\end{align}
and $s\to s^{\frac{k_0}{2}-1}\Upsilon_{k_0,i}(s)$ belongs to $L^1(0,r_0)$.

Next we show that for every $r \in (0,r_0)$
\begin{equation}\label{eq:9}
c_{k_0,i}(r)+\int_{0}^r t^{\frac{k_0}{2}-1}\Upsilon_{k_0,i}(t) \, dt=0.	
\end{equation}
We argue by contradiction assuming that there exists $r \in (0,r_0)$
such that \eqref{eq:9} does not hold.  Then by
\eqref{eq:13} and \eqref{eq:10}
\begin{equation}\label{eq:11}
  \varphi_{k_0,i}(\la) \sim
  \frac{k_0 \la^{-N+2-\frac{k_0}{2}}}{2(N+k_0-2)}
  \left(c_{k_0,i}(r)
    +\int_{\la}^r t^{\frac{k_0}{2}-1}\Upsilon_{k_0,i}(t) \, dt\right) \quad \text{as } \la \to 0^+.	
\end{equation}
From H\"older's inequality,  a change of variables, and \eqref{Hardy ineq}
  \[\int_0^{r_0} \la^{N-3}|\varphi_{k_0,i}(\la)|^2 \, d\la
  \le \int_0^{r_0} \la^{N-3}\left(\int_{\mathbb{S}^{N-1}}|U(\la\theta )|^2 dS\,\right) d\la=
  \int_{B_{r_0}}\frac{|U|^2}{|y|^2} dy<+\infty
\]
thus contradicting \eqref{eq:11}. Hence \eqref{eq:9} is proved.

Furthermore from \eqref{eq:19} and \eqref{eq:9}
\begin{align}\label{eq:12}
\bigg|\la^{-N+2-\frac{k_0}{2}}\bigg(c_{k_0,i}(r)&+\int_{\la}^r t^{\frac{k_0}{2}-1}\Upsilon_{k_0,i}(t) \, dt\bigg)\bigg|=\la^{-N+2-\frac{k_0}{2}}
\left|\int_0^{\la} t^{\frac{k_0}{2}-1}\Upsilon_{k_0,i}(t) \, dt\right| \\
&\le \la^{-N+2-\frac{k_0}{2}}\int_0^\la t^{N-2+k_0}\left|t^{-N+1-\frac{k_0}{2}}\Upsilon_{k_0,i}(t)\right| \, dt  \notag\\ 
&\le \la^\frac{k_0}{2}\int_0^\la\left|t^{-N+1-\frac{k_0}{2}}\Upsilon_{k_0,i}(t)\right| \, dt=O\left(\la^{\frac{4\epsilon}{N+2\epsilon}+\frac{k_0}{2}}\right) \quad \text{as } \la \to 0^+.\notag 
\end{align}	
Then the conclusion follows form \eqref{eq:13}, \eqref{eq:9}, and \eqref{eq:12}.
\end{proof}

\begin{proposition} \label{limit-gamma-positive}
Let $\gamma$ be as in \eqref{limit gamma }. Then 
\[\lim_{r \to 0^+} r^{2\gamma} H(r)>0.\]
\end{proposition}

\begin{proof}
  For any $\la \in (0,r_0)$ the function $U(\la \cdot)$ belongs to
  $L^2(\mathbb{S}^{N-1})$. Then we can expand it in Fourier series
  respect to the basis $\{Y_{k,i}\}_{k\in\mathbb{N}, i=1,\dots,N_k}$
  introduced in Proposition \ref{eigen}:
\[U(\la \, \cdot)=\sum_{k=0}^\infty\sum_{i=1}^{N_k}\varphi_{k,i}(\la)Y_{k,i}\quad \text{in } L^2(\mathbb{S}^{N-1}),\]
where we have defined $\varphi_{k,i}(\la)$ in \eqref{Fourier-coefficents} for any $k \in \mathbb{N}$ and any $i=1,\dots, N_k$.
From \eqref{mu-O}, a change of variables  and the Parseval identity
\begin{equation}\label{eq:15}
	H(\la)=(1+O(\la))\int_{\mathbb{S}^{N-1}}U^2 (\la \theta) \,dS=(1+O(\la))\sum_{k=0}^\infty\sum_{i=1}^{N_k}|\varphi_{k,i}(\la)|^2.
\end{equation}
We argue by contradiction assuming that
$\lim_{r \to 0^+} r^{2\gamma} H(r)=0$. Then by \eqref{eq:15}, letting
$k_0$ be as in 
\eqref{gamma as k0/2},
\begin{equation}\label{eq:16}
\lim_{\la \to 0^+}\la^{-\frac{k_0}{2}}\varphi_{k_0,i}(\la)=0 \quad \text{for any } i=1,\dots,N_{k_0}. 
\end{equation} 
From \eqref{Fourier-coefficents-expansion} it follows that 
\begin{align}\label{eq:18}
r^{-\frac{k_0}{2}}\varphi_{k_0,i}(r)&+\frac{2N+k_0-4}{2(N+k_0-2)}\int_{0}^r s^{-N+1-\frac{k_0}{2}}\Upsilon_{k_0,i}(s) \, ds\\
\notag&+\frac{k_0r^{-N+2-k_0}}{2(N+k_0-2)} \int_{0}^rs^{\frac{k_0}{2}-1} \Upsilon_{k_0,i}(s) \, ds=0
\end{align}
for any $r \in ( 0,r_0)$ and any $ i=1,\dots,N_{k_0}$.

In view of \eqref{wla}, \eqref{Fourier-coefficents}, \eqref{eq:19},
and \eqref{eq:12}, \eqref{eq:18} implies that
\begin{equation}\label{eq:46}
\sqrt{H(\lambda)} 
\int_{\mathbb{S}^{N-1}}W^\lambda Y_{k_0,i} \,dS
= \varphi_{k_0,i}(\lambda)=O\Big(\lambda^{\frac{4\epsilon}{N+2\epsilon}+\frac{k_0}2}\Big)\quad\text{as
  }\lambda\to0^+
\end{equation}
for all $i=1,\dots,N_{k_0}$. From \eqref{H lower bound} with
$\sigma=\frac{4\epsilon}{N+2\epsilon}$ we have that $\sqrt{H(\lambda)
}\geq
\sqrt{\alpha_{\frac{4\epsilon}{N+2\epsilon}}}\lambda^{\frac{k_0}2+\frac{2\epsilon}{N+2\epsilon}}$
in a neighbourhood of $0$, so that \eqref{eq:46} implies that 
\begin{equation}\label{eq:47}
\int_{\mathbb{S}^{N-1}}W^\lambda Y_{k_0,i} \,dS
=O\Big(\lambda^{\frac{2\epsilon}{N+2\epsilon}}\Big)=o(1)\quad\text{as
  }\lambda\to0^+
\end{equation}
for all $i=1,\dots,N_{k_0}$.

On the other hand, by
Proposition \ref{u-asymptotic} and 
continuity of the trace map $\gamma_1$ (see Proposition \ref{trace
  r}), for every sequence $\lambda_n\to0^+$, there exist a
  subsequence $\{\lambda_{n_k}\}$ and $\Psi\in \mathop{\rm span}\{
Y_{k_0,i}:m=i,\dots, N_{k_0}\}$ such that
  \begin{equation}\label{eq:48}
    \|\Psi\|_{L^2({\mathbb
        S}^{N-1})}=1\quad\text{and}
    \quad W^{\lambda_{n_k}}\to\Psi \quad\text{in }L^2({\mathbb S}^{N-1}).
\end{equation}
From (\ref{eq:47}) and (\ref{eq:48}) it follows that
\[
0=\lim_{k\to\infty}
\int_{\mathbb{S}^{N-1}}W^{\lambda_{n_k}} \Psi \,dS
=\|\Psi\|^2_{L^2({\mathbb
        S}^{N-1})}=1,
\]
thus reaching a contradiction.
\end{proof}

We are now ready to prove
he following result, which is a more complete version of
  Theorem~\ref{main-theorem}.

\begin{theorem} \label{main-theorem-precise}
 Let $N \ge3$ and let
  $u \in H^1(B_R\setminus \Gamma)$ be a non-trivial
  weak solution to
  \eqref{op}, with $\Gamma$ defined in  \eqref{Gamma
  2}--\eqref{Gamma} and $f$ satisfying either assumption \eqref{h1} or
  assumption \eqref{h2}. Then there exists $k_0\in \mathbb{N}$ such
  that,
letting $\mathcal N$ be as in Section \ref{sec:almgr-type-freq},  
\begin{equation}\label{limit-gamma-k0}
	\lim_{r \to 0^+}\mathcal{N}(r)=\frac{k_0}{2}.
\end{equation}
Moreover if $N_{k_0}$ is the multiplicity of the eigenvalue
$\mu_{k_0}$ of problem \eqref{sphereprob} and
$\{Y_{k_0,i}\}_{i=1,\dots,N_{k_0}}$ is a $L^2(\mathbb{S}^{N-1})$-orthonormal
basis of the eigenspace associated to 
$\mu_{k_0}$, then

  \begin{equation}\label{limit-u}
\lambda^{-\frac{k_0}{2}}u(\lambda \cdot) \to
          \Phi\quad\text{and}\quad \lambda^{1-\frac{k_0}{2}}
\left(\nabla_{B_R\setminus\Gamma}u\right)(\lambda
     \cdot)
     \to \nabla_{\R^N\setminus\tilde\Gamma}\Phi \quad
     \text{in } L^2(B_1) \quad \text{as } \la \to 0^+,
   \end{equation}
   where
   \begin{equation*}
     \Phi=
     \sum_{i=1}^{N_{k_0}}	\alpha_i Y_{k_0,i}\left(\frac{y}{|y|}\right)
   \end{equation*}
$(\alpha_1,\dots,\alpha_{N_{k_0}})\neq(0,\dots,0)$
 and, for all $i\in\{1,2,\dots,N_{k_0}\}$, 
\begin{multline}\label{alpha-i}
\alpha_i=r^{-k_0/2}\int_{\mathbb{S}^{N-1}} 	u(F(r\theta)) Y_{k_0,i}(\theta)\, dS\\
 +\frac{1}{2-N-k_0} \int_{0}^r\left(\frac{2-N-\frac{k_0}{2}}{s^{N+\frac{k_0}{2}-1}}-\frac{k_0s^{\frac{k_0}{2}-1}}{2r^{N-2+k_0}}\right)\Upsilon_{k_0,i}(s)\, ds
\end{multline}
for any $r \in (0,r_0)$ for some $r_0>0$, where we have defined
$\Upsilon_{k_0,i}$ in \eqref{resto} and $F$ is the diffeomorphism
introduced in Proposition \ref{diffeomorphism}.
\end{theorem}

\begin{proof}
  \eqref{limit-gamma-k0} directly comes from \eqref{gamma as k0/2}.
  Let  $U=u \circ F$ and
  $\{\la_n\}_{n \in \mathbb{N}}$ be a sequence such that
  $\lim_{n \to \infty}\la_n=0^+$. By Proposition \ref{u-asymptotic}
  and Proposition \ref{limit-gamma-positive} there exist a subsequence
  $\{\la_{n_k}\}_{k \in \mathbb{N}}$ and constants
  $\alpha_1,\dots,\alpha_{N_{k_0}}$ such that
  $(\alpha_1,\dots,\alpha_{N_{k_0}})\neq(0,\dots,0)$ and
  \[
    \la_{n_k}^{-\frac{k_0}{2}}U(\la_{n_k} y) \to |y|^{\frac{k_0}{2}} \sum_{i=1}^{N_{k_0}}	\alpha_i Y_{k_0,i}\left(\frac{y}{|y|}\right) 
  \quad \text{in } H^1(B_1 \setminus \tilde \Gamma) \quad \text{as } k
  \to \infty.\]
Now we show that the coefficients
$\alpha_1,\dots,\alpha_{N_{k_0}}$ do not depend on
$\{\la_n\}_{n \in \mathbb{N}}$ nor on its subsequence
$\{\la_{n_k}\}_{k \in \mathbb{N}}$. Thanks to the continuity of the
trace operator $\gamma_1$ introduced in Proposition \ref{trace r}
\[
   \la_{n_k}^{-\frac{k_0}{2}}U(\la_{n_k} \cdot) \to \sum_{i=1}^{N_{k_0}}	\alpha_i Y_{k_0,i} \quad \text{in }L^2(\mathbb{S}^{N-1}) \quad \text{as } k \to \infty\]
and therefore, letting $\varphi_{k_0,i}$ be as in
\eqref{Fourier-coefficents} for any $i=1, \dots, N_{k_0}$,
\[\lim_{k \to \infty}\la^{-\frac{k_0}{2}}\varphi_{k_0,i}(\la_{n_k})=
  \lim_{k \to \infty}\int_{\mathbb{S}^{N-1}}{\la_{n_k}^{-k_0/2}}U(\la_{n_k} \theta){Y_{k_0,i}(\theta)}\, dS
= \sum^{N_{k_0}}_{j=1}	\alpha_j \int_{\mathbb{S}^{N-1}}Y_{k_0,j}Y_{k_0,i}\, dS=\alpha_i.\]
On the other hand by \eqref{Fourier-coefficents-expansion}
\begin{multline*}
  \lim_{k \to \infty}\la^{-\frac{k_0}{2}}\varphi_{k_0,i}(\la_{n_k})
  =r^{-\frac{k_0}{2}}\varphi_{k_0,i}(r)+\frac{2N+k_0-4}{2(N+k_0-2)}\int_{0}^r s^{-N+1-\frac{k_0}{2}}\Upsilon_{k_0,i}(s) \, ds\\
+\frac{k_0r^{-N+2-k_0}}{2(N+k_0-2)} \int_{0}^rs^{\frac{k_0}{2}-1} \Upsilon_{k_0,i}(s) \, ds,
\end{multline*}
for all  $i=1, \dots, N_{k_0}$ and $r\in(0,r_0]$, where we have defined $\Upsilon_{k_0,i}$ in \eqref{resto}. We deduce that 
\begin{multline}\label{eq:23}
	\alpha_i=r^{-\frac{k_0}{2}}\varphi_{k_0,i}(r)+\frac{2N+k_0-4}{2(N+k_0-2)}\int_{0}^r s^{-N+1-\frac{k_0}{2}}\Upsilon_{k_0,i}(s) \, ds\\
 +\frac{k_0r^{-N+2-k_0}}{2(N+k_0-2)} \int_{0}^rs^{\frac{k_0}{2}-1} \Upsilon_{k_0,i}(s) \, ds
\end{multline}
and so $\alpha_i$ does not depend on  $\{\la_n\}_{n \in \mathbb{N}}$ nor on its subsequence $\{\la_{n_k}\}_{k \in \mathbb{N}}$ thus implying that
\begin{equation}\label{eq:22}
	\la^{-\frac{k_0}{2}}U(\la y) \to |y|^{\frac{k_0}{2}} \sum_{i=1}^{N_{k_0}}	\alpha_i Y_{k_0,i}\left(\frac{y}{|y|}\right) \quad \text{in } H^1(B_1 \setminus \tilde \Gamma)
 \quad \text{as } \la \to 0^+.
\end{equation}
To prove \eqref{limit-u} we note that
\[
  \la^{-\frac{k_0}{2}}u(\la x)=\la^{-\frac{k_0}{2}}U(\la
  G_\la(x)), \quad \nabla \left(\la^{-\frac{k_0}{2}}u(\la
    x)\right)=\nabla \left(\la^{-\frac{k_0}{2}}U(\la
    x)\right)(G_\la(x))J_{G_\lambda}(x),
\]
where $G_\la(x)=\frac{1}{\la}F^{-1}(\la x)$ and  $F$ is the
diffeomorphism introduced in  Proposition \ref{diffeomorphism}.
We also have by Proposition \ref{diffeomorphism} that 
\[
  G_\la(x)= x +O(\la) \quad \text{and} \quad J_G(x)=\mathop{\rm
    Id}\nolimits_N +O(\la)
\]
as $\la \to 0^+$ uniformly respect to $x \in B_1$. Then from  \eqref{eq:22} we deduce \eqref{limit-u} and   \eqref{alpha-i} follows from \eqref{eq:23} and \eqref{Fourier-coefficents}.
\end{proof}

\bibliographystyle{acm}
\bibliography{References}
\end{document}